\documentclass[12pt,oneside,reqno]{amsart}
\usepackage[margin=3cm]{geometry}

\usepackage{graphicx}
\graphicspath{{PW_CPR_Images/}{./}} % Specifies where to look for included images (trailing slash required)
\usepackage{float}
\usepackage{subfigure}
\usepackage[colorlinks=true, allcolors=blue]{hyperref}
\usepackage{multirow}%
\usepackage{amsmath,amssymb,amsfonts}%
\usepackage{amsthm}%
\usepackage{mathrsfs}%
\usepackage[title]{appendix}%
\usepackage{xcolor}%
\usepackage{textcomp}%
\usepackage{manyfoot}%
\usepackage{booktabs}%
\usepackage{algorithm}%
\usepackage{algorithmicx}%
\usepackage{algpseudocode}%??
\usepackage{listings}%
\usepackage{blindtext}
\usepackage{hyperref}
\usepackage{enumerate}%??
\usepackage{tikz}
\usetikzlibrary{shapes,snakes}

\usepackage{placeins} % ????????
\usepackage{caption} % ?????????
\newtheorem{theorem}{Theorem}[section]

\newtheorem{lemma}[theorem]{Lemma}%  
\newtheorem{example}[theorem]{Example}%
\newtheorem{remark}[theorem]{Remark}%
\newtheorem{proposition}[theorem]{Proposition}%

\newcommand{\Z}{\mathbb{Z}}
\newcommand{\R}{\mathbb{R}}
\newcommand{\C}{\mathbb{C}}

\numberwithin{equation}{section}

\title[]{Conjugate phase retrieval on graphs and with applications in shift-invariant spaces}
\author{Cheng Cheng, Baixiang Wu, Jun Xian}

\address{C. Cheng:
	School of Mathematics, Sun Yat-sen University, Guangzhou, Guangdong,
	510275, China, email: chengch66@mail.sysu.edu.cn}

\address{B. Wu:
	School of Mathematics, Sun Yat-sen University, Guangzhou, Guangdong,
	510275, China, email: wubx29@mail2.sysu.edu.cn}

\address{J. Xian: School of Mathematics and Guangdong Province Key Laboratory of Computational Science, Sun Yat-sen University, Guangzhou, Guangdong, 510275, China, email: xianjun@mail.sysu.edu.cn}
	
	\thanks{This project is partially supported by  National Key RD Program of China (No. 2024YFA1013703),  National Natural Science Foundation of China (12171490, 12371104), China Scholarship Council (202306380270, 202506380227), Fundamental Research Funds for the Central Universities, Sun Yat-sen University (24lgqb019-2), and the Guangdong Province Key Laboratory of Computational Science, China.}
\date{}
\begin{document}
	\maketitle
	
	\begin{abstract}
		In this paper, we study the conjugate phase retrieval for complex-valued \mbox{signals} residing on graphs, and explore its applications to shift-invariant spaces. Given a complex-valued graph signal $\bf f$ residing  on the graph $\mathcal G$, we introduce a graph ${\mathcal G}_{\bf f}$ and show that its connectivity is sufficient to determine 
		$\bf f$ up to a global unimodular constant and conjugation. We then construct two explicit graph models and show that graph signals residing on them can be recovered, up to a unimodular constant and conjugation, from its absolute values on the vertices and the relative magnitudes between neighboring vertices. 
	Building on this graph-based framework, we apply our results to shift-invariant spaces generated by real-valued functions. For signals in the Paley-Wiener space, we show that any complex-valued function can be recovered, up to a unimodular constant and conjugation, from structured phaseless samples taken at three times the Nyquist rate. For more general shift-invariant spaces, we establish the conjugate phase retrievability of signals from phaseless samples collected on a discrete sampling set, in conjunction with relative magnitude measurements between neighboring sample points. Two numerical reconstruction algorithms are  introduced to recover the signals in the Paley-Wiener space and general shift-invariant spaces, up to a unimodular constant and conjugation, from the given phaseless measurements.
		%A numerical reconstruction algorithm is also introduced to recover the signal, again up to a unimodular constant and conjugation.
	\end{abstract}

	\section{Introduction}
	
Phase retrieval has been studied in  various fields, including X-ray crystallography \cite{harrison1993phase,millane1990phase}, diffraction imaging \cite{bunk2007diffractive}, optics \cite{walther1963question}, and astronomical imaging \cite{dainty1987phase}.  It considers the recovery of a signal from the magnitudes of its linear measurements, up to a trivial ambiguity.  
	 In recent years, phase retrieval has seen significant advances, see \cite{alaifari2021uniqueness,balan2009painless,balan2006signal,bandeira2014saving,cahill2016phase,candes2015phase,candes2013phaselift,chen2022phase,chen2020phase,cheng2021stable,grochenig2020phase,jaming2014uniqueness,pohl2014phaseless,romero2021sign,thakur2011reconstruction} and references therein. 
	The shift-invariant space has been extensively studied in wavelet analysis and approximation theory \cite{aldroubi2000beurling, aldroubi2001nonuniform, aldroubi2005convolution,boor1994structure}. The phase retrieval problem in the shift-invariant space is an infinite-dimensional phaseless sampling and reconstruction problem. It considers the recovery of the signal from its phaseless samples taken on the domain or on a discrete sampling set \cite{chen2020phase, cheng2019phaseless,cheng2021stable2,grochenig2020phase,romero2021sign,shenoy2016exact,thakur2011reconstruction}. 
	Let $V^p(\phi)$ be the shift-invariant space with the real-valued generator $\phi$, 
	\begin{align}\label{Vspace.def}
		V^p(\phi) = \{\sum_{n\in \Z} c_n \phi(\cdot - n): c=(c_n)_{n\in \Z} \in \ell^p(\Z), 1\le p\le \infty  \}.
	\end{align}
	For a signal $f\in V^p(\phi)$, $e^{i\alpha_1}f$  or $e^{i\alpha_2}\overline{f}$
	%$e^{i\alpha}f$  or $e^{i\alpha}\overline{f}$
	can not be distinguished from the magnitudes of their pointwise evaluations for any $\alpha_1, \alpha_2 \in \R$ \cite{chen2024conjugate,chen2022phase,grochenig2020phase, lai2021conjugate,walther1963question}. 	
	The notion of complex conjugate phase retrieval in the complex conjugate invariant space $\mathcal{C}$ ($\overline{f} \in \mathcal{C}$ if $f\in \mathcal{C}$) has been introduced \cite{chen2024conjugate, chen2022phase,  lai2021conjugate}.  However, knowing only the  phaseless pointwise evaluations may not be sufficient to determine the signal, even up to a unimodular constant and conjugation. For example, let $\phi=\max(1-|x-1|,0)$, write $f=\sum_{k\in \Z} c_k\phi(\cdot-k)$ and $g=\sum_{k\in \Z} d_k\phi(\cdot-k)$, where the coefficients are
	\begin{align*}
		(c_0,c_1,c_2)=(1,i,1), \text{ and } c_k=0 \text{ for all } k\notin \{0,1,2\}
	\end{align*}
	and
	\begin{align*}
		(d_0,d_1,d_2)=(-1,i,1), \text{ and } d_k=0 \text{ for all } k\notin \{0,1,2\},  
	\end{align*}
	respectively. For any $x=t+k\in \R$   for some $t\in [0,1)$ and $k\in \Z$, we have
	\begin{align*}
		f(x)=tc_k+(1-t)c_{k-1} \text{ and } g(x)=td_k+(1-t)d_{k-1}.
	\end{align*}
	One may easily verify that
	\begin{align*}
		|f(x)|^2=|g(x)|^2, x\in \R.
	\end{align*}
	The authors considered the  (conjugate) phase retrieval of complex-valued  signals through structured phaseless measurements \cite{chen2024conjugate,chen2023conjugate, jaming2014uniqueness,lai2021conjugate,mc2004phase,pohl2014phaseless}. % such as magnitudes of the fractional Fourier transform \cite{jaming2014uniqueness}, and magnitudes of the  linear canonical transform \cite{chen2022phase2}. 
	In \cite{mc2004phase}, McDonald  proved that any complex-valued bandlimited signals can be determined, up to a unimodular constant and conjugation, from $|f(x)|$, $|f(x+c)-f(x)|, x\in \R$ for some $c>0$, on the whole real line,  which is a form of  phaseless structured convolutions.  The authors in \cite{evans2020conjugate} studied the conjugate phase retrieval  in finite-dimensional complex vector spaces using real frame. In \cite{lai2021conjugate}, the authors also considered the conjugate phase retrieval of complex-valued bandlimited signals from  phaseless structured convolutions and  proposed several algorithms. In  \cite{chen2023conjugate}, they established the  conjugate phase retrieval for functions living in the shift-invariant space generated by the  B-splines.   Later, one of the authors in this paper  showed that complex-valued signals in a Gaussian shift-invariant space can be determined, up to a unimodular   constant and conjugation, from phaseless Hermite samples \cite{chen2024conjugate}.

	Since the mapping from $\C$ to $\R^2$ is an isomorphism, conjugate phase retrieval for complex-valued signals can be regarded as a determination of the $\R^2$-valued function $(\Re f, \Im f)^{T}$, up to a $2\times 2$ orthogonal matrix. In \cite{chen2022phase}, the authors provide several characterizations of phase retrieval of  finite-dimensional vector fields on a graph from their absolute magnitudes at vertices and relative magnitudes between neighboring vertices. In this paper,  we study the conjugate phase retrieval of complex-valued graph signals. % indexed on a graph $\mathcal{G}=(V,E)$, where   $V$ is the  vertex set and  $E$ is the edge set, cf.  Example \ref{rem2}. %Write $\mathbf{f}= (f(x_n))_{x_n\in X}$, where $X$ is a set of sampling in \eqref{sampling.def.eq1} for the shift-invariant space $V^p(\phi)$ in \eqref{Vspace.def}. 
	Give a graph $\mathcal{G}=(V,E)$, we  introduce the graph $\mathcal{G}_{\bf f}$ associated with the complex-valued graph signal ${\bf f}:=(f_n)_{n\in V}$ which is indexed on the graph, and then use the connectivity of the graph $\mathcal{G}_{\bf f}$ to characterize the conjugate phase retrieval of the graph signal $\bf f$, see Theorem \ref{connect.pr.suff.thm1}. %In addition, if  ${\bf f}:=(f_n)_{n\in V}$ are the samples of the signal $f\in V^p(\phi)$ taken on a set of sampling $X$  satisfying \eqref{sampling.def.eq1}, applying the above characterization, 
	Later, we construct two explicit graphs and show that the graph signal  $\bf f$ can be determined, up to a unimodular constant and conjugation, from absolute magnitudes on the vertices and relative magnitudes between the neighboring vertices of some designed graphs, see Theorems \ref{thm.reference} and \ref{PW.adj.thm1}.

	The connectivity of the graph ${\mathcal G}_{\bf f}$ is sufficient for the conjugate phase retrieval  of the graph signal $\bf f$. If  the graph signal $\bf f$ is composed by  the values of a function $f$ taken on the discrete set $X$. Then the conjugate phase retrieval for graph signal $\bf f$ will be the conjugate phase retrieval of the function residing in some function spaces. We then further explore this in Sections \ref{cpr.pw.sec} and \ref{cpr.sis.sec} for functions residing in the  Paley-Wiener spaces and the shift-invariant spaces. % In \cite{mc2004phase}, Mc Donald  proved that any complex-valued bandlimited signals can be determined, up to a unimodular constant and conjugation, from $|f(x)|$, $|f(x+b)-f(x)|, x\in \R$ for some $b>0$, on the whole real line,  which is a form of structured convolutions. {\color{red} add the sampling density is 3 times, and also we have explicit reconstruction algorithm not relying on the analytic property of the signal.}
	For the Paley-Wiener space
	\begin{align}\label{PW.def}
		PW_B^{p}:= \{ f\in L^p(\R): \operatorname{supp} \ (\mathcal{F}f)\subset [-B,B]\}, 
	\end{align}
    where $p\in \{1,2\}$,
	we show that  $f\in PW_B^2$ can be determined, up to a  unimodular constant and conjugation, from the structured phaseless samples at three times the Nyqusit rate with an explicit reconstruction procedure, see Theorems  \ref{pw.2pts.thm1} and \ref{PW.adj.thm2}.

	%     By the open mapping theorem, it can be readily verified that complex-valued bandlimited functions can be determined, up to a unimodular constant, from their absolute values on an open subset of $\C$ \cite{wellershoff2024phase}. Jaming showed that complex-valued bandlimited functions can be determined, up to a unimodular constant, from absolute values on two parallel lines of $\C$ intersecting at an irrational angle \cite{jaming2014uniqueness}. Recently, Wellershoff show that complex-valued bandlimited functions can be determined, up to a unimodular constant, from absolute values on three parallel lines of $\C$, whose distances are rationally independent \cite{wellershoff2024phase}.

	A set $X:=\{x_n\}_{n\in \Z} \subset \R$ is said to be separated if
	\begin{align*}
		\inf_{n\ne m} |x_n-x_m|>0.
	\end{align*}
	A separated set $X:=\{x_n\}_{n\in \Z} \subset \R$ is called a (stable) set of sampling for $V^p(\phi)$  in \eqref{Vspace.def} if there exist constants $0<A\le B<\infty$ such that
	\begin{align}\label{sampling.def.eq1}
		A\|f\|_{L^p(\R)}^p\le \sum_{n\in \Z} |f(x_n)|^p \le B\|f\|_{L^p(\R)}^p, \ f\in V^p(\phi), 
	\end{align}
	where $1\le p\le \infty$. 
	A set $X$ is said to be a uniqueness set for $V^p(\phi)$ if $f\in V^p(\phi)$ and $f(x)= 0$ for all $x\in X$ implies that $f\equiv 0$. 
    There are  extensive studies that have demonstrated that every set of sampling for $V^p(\phi)$ satisfying \eqref{sampling.def.eq1} is a  uniqueness set and have considered the construction of the set of sampling for $V^p(\phi)$ satisfying \eqref{sampling.def.eq1}, see  \cite{aldroubi2000beurling, aldroubi2001nonuniform,aldroubi2005convolution,benedetto1992irregular, benedetto2001modern, grochenig2001foundations, grochenig2018sampling,higgins1996sampling,unser2000sampling,young2001introduction} and references therein.  
  In Section \ref{cpr.sis.sec}, we study the conjugate phase retrieval problem for complex-valued shift-invariant signals based on structured phaseless samples taken on a set of sampling.  In \cite{chen2024conjugate,chen2023conjugate,lai2021conjugate, mc2004phase}, they considered the conjugate phase retrieval from Hermite phaseless samples or structured convolutions  phaseless samples, and the proofs  utilize properties of  entire function of finite type. In this paper,  we derive theoretical characterizations for the conjugate phase retrieval of signals in the space $V^p(\phi)$ generated by the real-valued continuous function $\phi$,  see Theorems \ref{complete.cpr.thm1} and \ref{cpr.adj.thm}.   A key advantage of our framework is that, unlike many other approaches, our results do not rely on the analytic properties of entire functions, offering a more general solution.
   We show that when  the graph signal $\bf f$
 is formed by sampling a function  $f\in V^p(\phi)$ 
 on a set $X$ satisfying \eqref{sampling.def.eq1}, conjugate phase retrieval becomes feasible using absolute magnitudes assigned to graph vertices and relative magnitudes assigned to graph edges, as described in Theorems \ref{thm.reference} and \ref{PW.adj.thm1}. 
 Moreover, we not only establish the theoretical foundation for conjugate phase retrieval in shift-invariant spaces but also propose a practical reconstruction algorithm, see Algorithm \ref{alg2}. These ideas are further extended to the Paley-Wiener space, where we construct a corresponding reconstruction algorithm using similar arguments, see Algorithm \ref{alg1}.

	The short-time Fourier transform (STFT) simultaneously considers both the time and frequency domains, playing a crucial role in the time-frequency analysis \cite{alaifari2021uniqueness,grochenig2001foundations}. In \cite{alaifari2021uniqueness}, the authors showed that real-valued bandlimited signals can be determined, up to a sign, from phaseless STFT samples only in the time domain with a sampling rate twice the Nyquist rate. They further demonstrated that complex-valued bandlimited  signals can be uniquely determined, up to a unimodular constant, from STFT magnitudes,  under certain conditions on the  ambiguity function of the window function.  
	However, when reconstructing a complex-valued bandlimited function from phaseless STFT measurements taken only in the time domain with a real-valued window function, one cannot distinguish between the function  and its conjugation.  
In Section \ref{stft.pw.sec}, we consider the  conjugate phase retrieval of  bandlimited funtions from phaseless STFT measurements taken in the time domain only, assuming the window function $\psi$ is real-valued. We propose two phaseless STFT sampling schemes that enable conjugate phase retrieval under the condition that the Fourier transform of the window function  is nonzero almost everywhere, see Theorems \ref{thm4} and \ref{thm6}.

	\section{Conjugate phase retrieval of complex-valued signals residing on the graphs}\label{cpr.graph.sec}
	
	Phase retrieval for vector fields on graphs has been discussed in \cite{chen2022phase}. The authors there proposed a new paradigm of the phase retrieval, that is the determination of a vector field residing on the graph  $\mathcal{G}=(V, E)$, up to an orthogonal matrix, where $V$ is the vertex set and $E$ is the edge set. They showed that phase retrieval is possible for the vector field $\mathbf{f}_n, n\in V$ residing on the graph $\mathcal{G}=(V, E)$ from the absolute magnitudes $\|\mathbf{f}_n\|_2, n\in V$ and relative magnitudes $\|\mathbf{f}_n-\mathbf{f}_m\|_2, (n, m)\in E$, if the graph satisfies some conditions. For example, if  $\mathcal{G}_C=(V_C, E_C)$ is a complete graph, then phase retrieval is possible for the vector field $\mathbf{f}_n, n\in V_C$ from the the absolute magnitudes $\|\mathbf{f}_n\|_2, n\in V_C$ and relative magnitudes $\|\mathbf{f}_n-\mathbf{f}_m\|_2, (n, m)\in E_C$, see the following Proposition.  
	
	\smallskip
	\begin{proposition}\label{complete.prop.jfa}\cite[Theorem 4.1]{chen2022phase}
		Let $\mathcal{G}_C=(V_C, E_C)$ be a complete graph, and $\mathbf{f}=(\mathbf{f}_n)_{n\in V_C}, \mathbf{f}_n\in \R^d$ be a vector field on the vertex $V_C$. Then $\mathbf{f}$  can be determined, up to a $d\times d$ orthogonal matrix, from the absolute magnitudes $\|\mathbf{f}_n\|_2, n\in V_C$ and relative magnitudes $\|\mathbf{f}_n-\mathbf{f}_m\|_2, (n, m)\in E_C$.
	\end{proposition}

	\smallskip
	
		Let ${\mathcal G}_C=(V_C, E_C)$ be a complete graph with vertex set $V_C$ and edge set $E_C$, and ${\bf f}=({\bf f}_n)_{n\in V_C}\in \R^{2\times V_C}$ be a vector field residing on the graph ${\mathcal G}_C$.  Combine with  Proposition \ref{complete.prop.jfa}, one may easily verify that  the vector field $\bf f$ can be determined, up to a $2\times 2$ orthogonal matrix, from the absolute magnitudes taken on the vertices $\|{\bf f}_n\|_2, n\in V_C$ and relative magnitudes between the neighbors $\|{\bf f}_n - {\bf f}_m\|_2, (n,m) \in E_C$.

	Define the  isomorphism %linear correspondence
	\begin{align*}%\label{map.def}
		M:\ {\mathcal C}\ni f\longmapsto \begin{pmatrix} \Re f \\ \Im f
		\end{pmatrix}\in  \Re({\mathcal C})\times \Re({\mathcal C})
	\end{align*}
	between  complex-valued function in the  complex  conjugate invariant space $\mathcal C$ 
	and ${\mathbb R}^2$-valued function, where $\Re ({\mathcal C})$ is the space of real-valued function  in ${\mathcal C}$.   As a $2\times 2$ orthogonal matrix operating on the  vector-valued function  $\begin{pmatrix} \Re f \\ \Im f
	\end{pmatrix}$ corresponds to the conjugation or rotation  of the complex-valued function $f$, 
	the  conjugate phase retrieval problem for a complex-valued function    is 
	the determination of  a ${\mathbb R}^2$-valued function, up to a $2\times 2$ orthogonal matrix, from
	the structured phaseless samples taken on the graph. 
	%\smallskip

	\smallskip
	
	Now, we will study the determination of a complex-valued graph signal, up to a unimodular and conjugation, from its absolute magnitudes and relative magnitudes residing on  an {\em arbitrary} simple graph. 
	Let  ${\mathcal G}=(V, E)$ be a simple graph, and  ${\bf f}=(f_n)_{n\in V}, f_n\in \C$ be a complex-valued  signal residing on the graph. 	We say an edge $(n,m)\in E$ is  non-collinear if 
	\begin{equation}\label{noncollinear.def}
		{f}_n \overline{{f}_m} \ne  {f}_m \overline{f}_n. \end{equation}
	Equivalently, signals $f_n, f_m$  do not lie on the same line passing through the origin in the complex plane.
	 Denote  a complete subgraph of order $3$ by 
	${\mathcal G}_\triangle=(V_\triangle, E_\triangle)$.   Let $V_{\bf f}$ be the vertex set containing all complete subgraphs ${\mathcal G}_\triangle$ of order $3$
    %triangles 
    in the graph $\mathcal G$  and  $E_{\bf f}$ be
	the edge set containing   all pairs  of
	complete subgraphs  ${\mathcal G}_{\triangle_n}=(V_{\triangle_n}, E_{\triangle_n}), {\mathcal G}_{\triangle_m}=(V_{\triangle_m}, E_{\triangle_m}) \in V_{\bf f}$ satisfying the  common edge is non-collinear, that is, 
	$$f_i\overline{f_j}\neq f_j \overline{f_i}, \ \  (i, j)\in E_{\triangle_n}\cap  E_{\triangle_m}. $$ 
	Define 
	\begin{equation}\label{gf.def}
		{\mathcal G}_{\bf f}=(V_{\bf f}, E_{\bf f})
	\end{equation}
	be the graph with vertex set $V_{\bf f}$ and edge set  $E_{\bf f}$. 
	One may easily verify that the graph  ${\mathcal G}_{\bf f}=(V_{\bf f}, E_{\bf f})$ is defined if ${\bf f}=(f_n)_{n\in V}$ is provided. In the next theorem, we show that if ${\mathcal G}_{\bf f}$ is connected, then the graph signal  ${\bf f}=(f_n)_{n\in V}$ residing on the graph $\mathcal G$ can be determined, up to a unimodular constant and conjugation,  from the absolute magnitudes $|f_n|, n\in V$ and relative magnitudes  $|f_n-f_m|, (n, m)\in E$.

	\begin{theorem}\label{connect.pr.suff.thm1}
		Let $\mathcal{G}=(V, E)$ be a simple graph and  $\mathbf{f}= (f_n)_{n\in V}$ be a complex-valued graph signal residing on the graph. If for any vertex  $n\in V$ there exists a complete subgraph $\mathcal{G}_{\triangle_n}=(V_{\triangle_n},E_{\triangle_n})$ of order $3$ containing it,  and $\mathcal{G}_{\bf f}=(V_{\bf f},E_{\bf f})$ in \eqref{gf.def} is connected, then the graph signal $\mathbf{f}$ can be determined, up to unimodular constant and conjugation, from absolute magnitudes $|f_n|, n\in V$ and relative magnitudes $|f_n-f_m|, (n, m)\in E$. % Moreover, if $X=\{x_n\}_{n\in \Z} \subset \R$ be a set of sampling for the signal  $f\in PW_B^2$, then  $f$ can be determined, up to a unimodular constant and conjugation.
	\end{theorem}
	\begin{proof}
		Let $\mathbf{g}=({g}_n)_{n\in V}$ be a complex-valued vector indexed on the vertex set $V$ satisfying 
		\begin{align}\label{eq04}
			|{f}_n| = |{g}_n|
		\end{align}
		and
		\begin{align*}
			|{f}_n - {f}_m| = |{g}_n - {g}_m|
		\end{align*}
		for all $n\in V$ and $(n,m)\in E$.
		Then we have
		\begin{align}\label{eq54}
			\Re f_n \Re f_m + \Im f_n \Im f_m
			=\Re g_n \Re g_m + \Im g_n \Im g_m
		\end{align}
		for all $(n,m)\in E$. 
	Let $\mathcal{G}_{\triangle_p}=(V_{\triangle_p},E_{\triangle_p})$  be a  complete subgraph  of order 3, with vertex set $V_{\triangle_p}=\{n_p, m_p, k_p\}$ and  $(n_p, m_p)\in E_{\triangle_p}$ be the non-collinear edge satisfying \eqref{noncollinear.def}. Then 
		\begin{align}\label{eq05}
			\Re f_{n_p} \Im f_{m_p} \ne \Re f_{m_p} \Im f_{n_p},
		\end{align}
		which implies $f_{n_p}, f_{m_p}\ne 0$. Without loss of generality, assume $f_{n_p}=g_{n_p}>0$, otherwise replacing $\bf f$ by $ e^{i\alpha_1}\mathbf{f}$ and $\bf g$ by $ e^{i\alpha_2}\mathbf{g} $ respectively, where $\alpha_1, \alpha_2\in \R$. Combine with \eqref{eq54}, we have $\Re f_{m_p}= \Re g_{m_p}$ and then $\Im f_{m_p}= \tau_p \Im g_{m_p}\ne 0$ for some $\tau_p=\pm 1$ by \eqref{eq04} and \eqref{eq05}. Without loss of generality, assume $\tau_p= 1$.
 		As $\mathcal{G}_{\triangle_p}$ is a complete subgrapgh and by \eqref{eq54}, we then have
		\begin{equation*}
			\left\{ \begin{array}{ll}
				
				( \Re f_{k_p} - \Re g_{k_p} )  \Re f_{m_p}
				+ ( \Im f_{k_p} -  \Im g_{k_p} )  \Im f_{m_p} {}& = 0 \\
				( \Re f_{k_p} - \Re g_{k_p} )  \Re f_{n_p}
				+ ( \Im f_{k_p} -  \Im g_{k_p} )  \Im f_{n_p} {}& = 0. 
			\end{array}\right.
		\end{equation*}
		Together with \eqref{eq05}, we have $\Re f_{k_p} = \Re g_{k_p}$ and $\Im f_{k_p} =  \Im g_{k_p}$. Thus, we have
		\begin{align*}
			f_n = {g}_n, n\in V_{\triangle_p}.
		\end{align*} 
		It suffices to show that $f_n = {g}_n$ holds for all $n\in V$. 
		As for any $n\in V$, there exists a complete subgraph $\mathcal{G}_{\triangle}=(V_{\triangle},E_{\triangle})$ of order $3$ in $V_{\textnormal{{\bf f}}}$ such that $n\in V_{\triangle}$. Combine this with the definition and connectivity of $\mathcal{G}_{\bf f}$, it suffices to prove that
		\begin{align}\label{eq89}
			{f}_n = {g}_n, n\in V_{\triangle_p} \cup V_{\triangle_q},
		\end{align}
		where $\mathcal{G}_{\triangle_q}=(V_{\triangle_q},E_{\triangle_q})$ is a complete subgraph of order 3 with $V_{\triangle_q}=\{n_q, m_q, k_q\}$, and $(\mathcal{G}_{\triangle_p}, \mathcal{G}_{\triangle_q})\in E_{\bf f}$. %that has a  non-collinear edge shared with $\mathcal{G}_{\triangle_p}$.
		Without loss of generality, let the common non-collinear edge between $\mathcal{G}_{\triangle_p}$ and $\mathcal{G}_{\triangle_q}$ be $(n_p,m_p)=(n_q,m_q)$. 
		%As   $(n_p,m_p)$ is a non-collinear edge, we have 
		% Then $\mathbf{f}_{n_2} = \mathbf{f}_{n_1} = \mathbf{g}_{n_1} = \mathbf{g}_{n_2}$, $\mathbf{f}_{m_2} = \mathbf{f}_{m_1} = \mathbf{g}_{m_1} = \mathbf{g}_{m_2}$ and
		%\begin{align}\label{eq88}
		%	\Re f_{n_p} \Im f_{m_p} \ne \Re f_{m_p} \Im f_{n_p}. 
		%\end{align}
		% by $\mathbf{f}_{n_2} \overline{\mathbf{f}_{m_2}} \ne \mathbf{f}_{m_2} \overline{\mathbf{f}_{n_2}}$. 
	 Combine with that $\mathcal{G}_{\triangle_q}$ is a complete subgraph and \eqref{eq54}, we have
		\begin{equation*}
			\left\{
			\begin{array}{ll}
				      (\Re f_{k_q} - {R} g_{k_q} ) \Re f_{n_p} + (\Im  f_{k_q} - \Im  g_{k_q} ) \Im  f_{n_p} {}&  = 0 \\
				   (\Re f_{k_q} - \Re g_{k_q} ) \Re f_{m_p} + (\Im  f_{k_q} - \Im  g_{k_q} ) \Im  f_{m_p} {}& = 0.
			\end{array}\right.
		\end{equation*}
		Together with \eqref{eq05}, we have $\Re f_{k_q} - \Re g_{k_q}=0$ and $\Im  f_{k_q} - \Im  g_{k_q}=0$, that is, ${f}_{k_q} = {g}_{k_q}$. Therefore, \eqref{eq89} holds and thus $\mathbf{f}$ is determined, up to a unimodular constant and conjugation. 
	\end{proof}

\begin{remark} 
    Let $\mathcal{G}=(V,E)$ be a simple graph with complex-valued graph signal $\mathbf{f}=(f_n)_{n\in V}$. Let $\mathcal{G}_{\bf f} = (V_{\bf f},E_{\bf f})$ be the graph as in \eqref{gf.def}.
    For a complete subgraph $\mathcal{G}_\triangle=(V_\triangle,E_\triangle)$ of order $3$ in the vertex set $V_{\bf f}$, the authors in \cite{chen2022phase} assumed that the graph signal $f_n, n\in V_\triangle$, residing on the complete subgraph $\mathcal{G}_\triangle=(V_\triangle,E_\triangle)$ is affinely independent.  Equivalently,  $f_n, n\in V_\triangle$ form a non-degenerate triangle in the complex plane, that is, they do not lie on a common line. In this paper, we dispense with this restriction, and the conclusion remains valid even if $f_n, n\in V_\triangle$, lie on the same straight line.   This improvement holds by assuming the signals residing on the endpoints of the common edge of two complete subgraphs  do not pass through the origin. % which is true by the construction of the graph $G_{\bf f}$ in \eqref{gf.def}.
    %In this paper, the condition that $(f_i)_{i\in V_\triangle}$ do not lie on a common line can be dropped, see Theorem \ref{connect.pr.suff.thm1}.
\end{remark}

	Theorem \ref{connect.pr.suff.thm1} shows that  a graph signal $\bf f$ can be determined, up to a unimodular constant and conjugation, from absolute magnitudes on the vertices and relative magnitudes on the neighboring vertices.  It is natural to consider if we can construct some graphs $\mathcal G$ explicitly  such that the signal $\bf f$ can be determined, up to a unimodular constant and conjugation, from absolute magnitudes on the vertices and relative magnitudes on the edges of the proposed graph $
	\mathcal G$.  Our first result shows that the graph ${\mathcal G}_{\bf f}$ is connected if the graph signal $\bf f$ resides on the graph $\mathcal G$ where every vertex is connected to two reference points. Hence the graph signal $\bf f$  can be determined, up to a unimodular constant and conjugation, from the absolute magnitudes on the vertices and relative magnitudes on the edges  by Theorem \ref{connect.pr.suff.thm1}.
	
	\begin{theorem}\label{thm.reference}
		Let $\mathcal{G} = (V, E)$ be a graph with the vertex set $V$ and edge set
		\begin{align*}
			E=\{(n,q):n\in V,q=k,l \}
		\end{align*}
		for some $k\ne l\in V$, and let
		$\mathbf{f}= (f_n)_{n\in V}$ be a complex-valued signal residing on the graph. If the edge $(k, l)$ is non-collinear satisfying \eqref{noncollinear.def}, that is  $f_k\overline{f_l}\neq f_l \overline{f_k}$, then $\mathcal{G}_{\bf f} = (V_{\bf f}, E_{\bf f})$ in \eqref{gf.def} corresponding to  $\mathbf{f}$ is connected. Furthermore,  $\bf f$ can be determined, up to a unimodular constant and conjugation, from absolute magnitudes $ |f_n|, n\in V $ and relative magnitudes $|f_n-f_m|, (n,m)\in E$.
	\end{theorem}

	From the above theorem, we know that a graph signal is conjugate phase retrieval  if all its components are connected to two fixed reference vertices. We now extend this concept from a centralized, reference-based structure to a distributed one by considering the circulant graph. 
	
	%
	%Next, we show that the complex-valued vector $\mathbf{f}=(f_n)_{n\in V}$ indexed on the circulant graph $\mathcal G=(V, E,Q)$ can be determined, up to a unimodular constant and conjugation, from absolute magnitudes $|f_n|, n\in V$ and relative magnitudes $|f_n - f_m|, (n,m)\in E$, where $\mathcal G=(V, E,Q)$ is the circulant graph defined as in \ref{circulant.edge}. %Furthermore, $f$ can be determined, up to a unimodular constant and conjugation, from these magnitudes, assuming $X$ is a set of sampling for $V^p(\phi)$.
	
	\begin{theorem}\label{PW.adj.thm1}
		Let $\mathcal{G} = (V, E,Q)$ be a circulant graph with the vertex set $V$ and the edge set
		\begin{align*}%\label{circulant.edge}
			E=\{(n,n\pm q): n\in V, q\in Q \},
		\end{align*} 
		where $Q=\{1,2\}$. Let
		$\mathbf{f}= (f_n)_{n\in V}$ be a complex-valued signal residing on the graph.  If $f_n\overline{f_{n+1}}\neq f_{n+1} \overline{f_n}, n\in V$, then $\mathcal{G}_{\bf f} = (V_{\bf f},E_{\bf f})$ in \eqref{gf.def} corresponding to  $\mathbf{f}$ is connected. Furthermore,  $\bf f$ can be determined, up to a unimodular constant and conjugation, from absolute magnitudes $ |f_n|, n\in V$ and relative magnitudes $|f_{n}-f_m|, (n,m)\in E$.
	\end{theorem}

	To illustrate these theoretical results, we now present concrete examples demonstrating the applications of Theorems \ref{thm.reference} and \ref{PW.adj.thm1}.

	\begin{example}\label{rem2}{\rm
			%	Let $f$ be a complex-valued function in $V^p(\phi)$ and  $X=\{x_n\}_{n\in \Z}$ be a set of sampling for $V^p(\phi)$ satisfying \eqref{sampling.def.eq1} and $\mathbf{f} = (f(x_n))_{n\in \Z}$.
			(a)  Let ${\mathcal G}_1=(V, E_1)$ be a graph   with the vertex set $V$ and edge set
			\begin{align*}
				E_1=\{(n,q):n\in V, q=k,l \}, 
			\end{align*}
			where $k\ne l\in V$ are two reference vertices, see  the plot in the right of Fig. \ref{fig1}.  By the construction of the graph ${\mathcal G}_1$,  every vertex $n\in V\setminus \{k, l\}$ is contained in a  complete subgraph  $\mathcal{G}_{1,\triangle_n} = (V_{1,\triangle_n}, E_{1,\triangle_n})$ of order $3$  with the vertex set
			\begin{align*}
				V_{1,\triangle_n} = \{n, k, l \}
			\end{align*}
			and edge set
			\begin{align*}
				E_{1,\triangle_n} = \{(i,j): i,j \in V_{1,\triangle_n}\}.
			\end{align*}
			Let $\mathbf{f}=(f_n)_{n\in V}$ be a graph signal on ${\mathcal G}_1$ satisfying $ f_k\overline{f_l}\ne f_l\overline{f_k}$. Then  the graph $\mathcal{G}_{1,\mathbf{f}} = (V_{1,\mathbf{f}}, E_{1,\mathbf{f}})$ in \eqref{gf.def} with the vertex set
			\begin{align*}
				V_{1,\mathbf{f}} = \{\mathcal{G}_{1,\triangle_n}: n\in V\setminus \{k,l\} \}
			\end{align*}
			and edge set
			\begin{align*}
				E_{1,\mathbf{f}} = \{(\mathcal{G}_{1,\triangle_n}, \mathcal{G}_{1,\triangle_{m}}): n, m\in V\setminus \{k,l\} \},
			\end{align*}
			is connected, cf. Theorem \ref{thm.reference}.  %	If we further assume that $\bf f$ is composed of values of function $f\in V^p(\phi)$ taken on the set of sampling $X$, then  $f$ can be determined, up to a unimodular constant and conjugation, from the absolute magnitudes  $ |f(x_n)|, n\in V$ and relative magnitudes $  |f(x_{k})-f(x_n)| $ and $ |f(x_{l})-f(x_n)|$ on edges in $E_1$, cf. Theorem \ref{thm.reference}. 
		
		\medskip	
			(b)   Let $\mathcal{G}_2=(V,E_2)$ be the graph  with the vertex set $V$ and edge set
			\begin{align*}
				E_2=\{(n,n\pm q): n\in V, q=1,2 \}, 
			\end{align*}
			see the  plot in the left of  Fig. \ref{fig1}.  Then every $n\in V$ is contained in a complete subgraph $\mathcal{G}_{2,\triangle_n} = (V_{2,\triangle_n}, E_{2,\triangle_n})$ of order $3$ with the vertex set
			\begin{align*}
				V_{2,\triangle_n} = \{n-1, n, n+1 \}
			\end{align*}
			and edge set
			\begin{align*}
				E_{2,\triangle_n} = \{(i,j): i,j \in V_{2,\triangle_n}\}.
			\end{align*}
			Let $\mathbf{f}=(f_n)_{n\in V}$ be a graph signal on ${\mathcal G}_2$ satisfying $ f_{n+1}\overline{f_n}\ne f_{n}\overline{f_{n+1}}$ for all $n\in V$. Then  the graph ${\mathcal G}_{2,\bf f} = (V_{2,\mathbf{f}}, E_{2,\mathbf{f}})$ in \eqref{gf.def} with  the vertex set 
			\begin{align*}
				V_{2,\mathbf{f}} = \{\mathcal{G}_{2,\triangle_n}: n\in V \}
			\end{align*}
			and edge set
			\begin{align*}
				E_{2,\mathbf{f}} = \{(\mathcal{G}_{2,\triangle_n}, \mathcal{G}_{2,\triangle_{n+1}}): n\in V \}
			\end{align*}
			is connected, cf. Theorem \ref{PW.adj.thm1}. %	If we further assume that $\bf f$ is composed of the values of the function $f\in V^p(\phi)$ taken on the set of sampling $X$,  then  $f$ can be determined, up to a unimodular constant and conjugation, from the absolute magnitudes  $ |f(x_n)|, n\in V$ and relative magnitudes $  |f(x_{k})-f(x_n)| $ and $ |f(x_{l})-f(x_n)|$ on edges in $E_2$, 	cf. Theorem \ref{PW.adj.thm1}.  %This is consistent with the conclusion in Theorem \ref{cpr.adj.thm}. 
			
			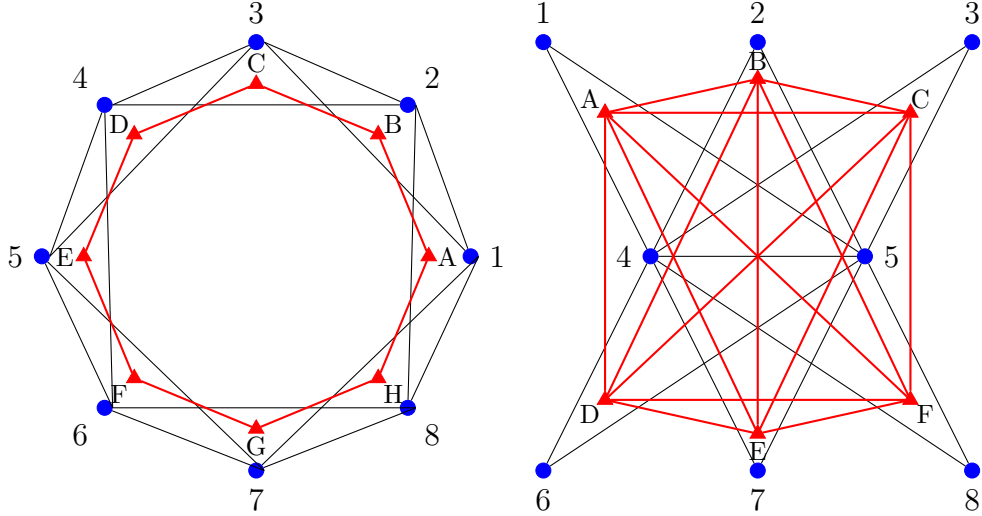
\begin{figure}[t]
				\centering
				\begin{minipage}{0.47\textwidth}
					\centering
					\begin{tikzpicture}[scale=0.95]
						
						% 定义圆的半径和节点数量
						\def\radius{3} % 圆的半径
						\def\nodes{8} % 节点数量

						\node[circle, draw=blue, fill=blue, inner sep=2pt, label=right:{$1$}] (N1) at ({360/\nodes * (1-1)}:\radius) {};
						\node[circle, draw=blue, fill=blue, inner sep=2pt, label=above right:{$2$}] (N2) at ({360/\nodes * (2-1)}:\radius) {};
						\node[circle, draw=blue, fill=blue, inner sep=2pt, label=above:{$3$}] (N3) at ({360/\nodes * (3-1)}:\radius) {};
						\node[circle, draw=blue, fill=blue, inner sep=2pt, label=above left:{$4$}] (N4) at ({360/\nodes * (4-1)}:\radius) {};
						\node[circle, draw=blue, fill=blue, inner sep=2pt, label=left:{$5$}] (N5) at ({360/\nodes * (5-1)}:\radius) {};
						\node[circle, draw=blue, fill=blue, inner sep=2pt, label=below left:{$6$}] (N6) at ({360/\nodes * (6-1)}:\radius) {};
						\node[circle, draw=blue, fill=blue, inner sep=2pt, label=below:{$7$}] (N7) at ({360/\nodes * (7-1)}:\radius) {};
						\node[circle, draw=blue, fill=blue, inner sep=2pt, label=below right :{$8$}] (N8) at ({360/\nodes * (8-1)}:\radius) {};

						\foreach \i in {1,...,\nodes} {
							\pgfmathsetmacro{\j}{mod(\i,\nodes)+1} 
							\pgfmathsetmacro{\k}{mod(\i+1,\nodes)+1} 
							\draw (N\i) -- (N\j); 
							\draw (N\i) -- (N\k); 
						}

						\node[regular polygon, regular polygon sides=3, draw=red, fill=red, scale=0.3, label={[shift={(0.2,-0.25)}]\footnotesize{B}}] (A1) at (barycentric cs:N1=1,N2=1,N3=1) {};
						\node[regular polygon, regular polygon sides=3, draw=red, fill=red, scale=0.3, label={[shift={(0,-0.1)}]\footnotesize{C}}] (A2) at (barycentric cs:N2=1,N3=1,N4=1) {};
						\node[regular polygon, regular polygon sides=3, draw=red, fill=red, scale=0.3, label={[shift={(-0.2,-0.25)}]\footnotesize{D}}] (A3) at (barycentric cs:N3=1,N4=1,N5=1) {};
						\node[regular polygon, regular polygon sides=3, draw=red, fill=red, scale=0.3, label={[shift={(-0.25,-0.4)}]\footnotesize{E}}] (A4) at (barycentric cs:N4=1,N5=1,N6=1) {};
						\node[regular polygon, regular polygon sides=3, draw=red, fill=red, scale=0.3, label={[shift={(-0.2,-0.6)}]\footnotesize{F}}] (A5) at (barycentric cs:N5=1,N6=1,N7=1) {};
						\node[regular polygon, regular polygon sides=3, draw=red, fill=red, scale=0.3, label={[shift={(0,-0.62)}]\footnotesize{G}}] (A6) at (barycentric cs:N6=1,N7=1,N8=1) {};
						\node[regular polygon, regular polygon sides=3, draw=red, fill=red, scale=0.3, label={[shift={(0.2,-0.6)}]\footnotesize{H}}] (A7) at (barycentric cs:N7=1,N8=1,N1=1) {};
						\node[regular polygon, regular polygon sides=3, draw=red, fill=red, scale=0.3, label={[shift={(0.25,-0.4)}]\footnotesize{A}}] (A8) at (barycentric cs:N8=1,N1=1,N2=1) {};

						\draw[-, red, thick] (A1) -- (A2);
						\draw[-, red, thick] (A2) -- (A3);
						\draw[-, red, thick] (A3) -- (A4);
						\draw[-, red, thick] (A4) -- (A5);
						\draw[-, red, thick] (A5) -- (A6);
						\draw[-, red, thick] (A6) -- (A7);
						\draw[-, red, thick] (A7) -- (A8);
						\draw[-, red, thick] (A8) -- (A1);

					\end{tikzpicture}
					
				\end{minipage}
				\hspace{-0.06\textwidth}
				\begin{minipage}{0.47\textwidth}
					\centering
					\begin{tikzpicture}[scale=0.95]		
						\node[circle, draw=blue, fill=blue, inner sep=2pt, label=above:{1}] (N1) at (-3, 3) {};
						\node[circle, draw=blue, fill=blue, inner sep=2pt, label=above:{2}] (N2) at (0, 3) {};
						\node[circle, draw=blue, fill=blue, inner sep=2pt, label=above:{3}] (N3) at (3, 3) {};

						\node[circle, draw=blue, fill=blue, inner sep=2pt, label=left:{4}] (N4) at (-1.5, 0) {};
						\node[circle, draw=blue, fill=blue, inner sep=2pt, label=right:{5}] (N5) at (1.5, 0) {};

						\node[circle, draw=blue, fill=blue, inner sep=2pt, label=below:{6}] (N6) at (-3, -3) {};
						\node[circle, draw=blue, fill=blue, inner sep=2pt, label=below:{7}] (N7) at (0, -3) {};
						\node[circle, draw=blue, fill=blue, inner sep=2pt, label=below:{8}] (N8) at (3, -3) {};

						\draw (N4) -- (N1);
						\draw (N4) -- (N2);
						\draw (N4) -- (N3);
						\draw (N5) -- (N1);
						\draw (N5) -- (N2);
						\draw (N5) -- (N3);

						\draw (N4) -- (N6);
						\draw (N4) -- (N7);
						\draw (N4) -- (N8);
						\draw (N5) -- (N6);
						\draw (N5) -- (N7);
						\draw (N5) -- (N8);

						\draw (N4) -- (N5);

						\node[regular polygon, regular polygon sides=3, draw=red, fill=red, scale=0.3, shift={(-3.6,3.2)}, label={[shift={(-0.2,-0.2)}]\footnotesize{A}}] (A1) at (barycentric cs:N1=1,N4=1,N5=1) {};
						\node[regular polygon, regular polygon sides=3, draw=red, fill=red, scale=0.3, shift={(0,4.7)}, label={[shift={(0,-0.15)}]\footnotesize{B}}] (A2) at (barycentric cs:N2=1,N4=1,N5=1) {};
						\node[regular polygon, regular polygon sides=3, draw=red, fill=red, scale=0.3, shift={(3.6,3.2)}, label={[shift={(0.13,-0.2)}]\footnotesize{C}}] (A3) at (barycentric cs:N3=1,N4=1,N5=1) {};
						\node[regular polygon, regular polygon sides=3, draw=red, fill=red, scale=0.3, shift={(-3.6,-3.2)}, label={[shift={(-0.2,-0.57)}]\footnotesize{D}}] (A6) at (barycentric cs:N6=1,N4=1,N5=1) {};
						\node[regular polygon, regular polygon sides=3, draw=red, fill=red, scale=0.3, shift={(0,-4.7)}, label={[shift={(0,-0.62)}]\footnotesize{E}}] (A7) at (barycentric cs:N7=1,N4=1,N5=1) {};
						\node[regular polygon, regular polygon sides=3, draw=red, fill=red, scale=0.3, shift={(3.6,-3.2)}, label={[shift={(0.2,-0.57)}]\footnotesize{F}}] (A8) at (barycentric cs:N8=1,N4=1,N5=1) {};
						
						\draw[-, red, thick] (A1) -- (A2);
						\draw[-, red, thick] (A1) -- (A3);
						\draw[-, red, thick] (A1) -- (A6);
						\draw[-, red, thick] (A1) -- (A7);
						\draw[-, red, thick] (A1) -- (A8);
						\draw[-, red, thick] (A2) -- (A3);
						\draw[-, red, thick] (A2) -- (A6);
						\draw[-, red, thick] (A2) -- (A7);
						\draw[-, red, thick] (A2) -- (A8);
						\draw[-, red, thick] (A3) -- (A6);
						\draw[-, red, thick] (A3) -- (A7);
						\draw[-, red, thick] (A3) -- (A8);
						\draw[-, red, thick] (A6) -- (A7);
						\draw[-, red, thick] (A6) -- (A8);
						\draw[-, red, thick] (A7) -- (A8);
					\end{tikzpicture}

				\end{minipage}

				\caption{
					Plotted on the left are the graph $\mathcal{G}_2=(V, E_2)$, constructed in Example \ref{rem2} (b),   with vertices marked by blue dots and edges marked in black lines. Let $\mathbf{f}=(f_n)_{n\in V}$ be a complex-valued graph signal on $\mathcal{G}_2=(V, E_2)$. Then graph $\mathcal{G}_{2, {\bf f}}$ is connected  with  vertices from $A$ to $H$ corresponding to $\{\triangle 812, \triangle 123, \triangle 234, \triangle 345, \triangle 456, \triangle 567, \triangle 678, \triangle 781\}$, marked in red  filled triangles, and with  the undirected edges $\{(A, B), (B,C), (C, D), (D,E), (E, F), (F, G), (G, H), (H, A) \}$, marked in red solid lines, where the edge $(n,n+1)\in E_2$ does not pass the origin  in the complex plane for every $n\in V$. \\
                    %$f_n$ and $f_{n+1}$ are non-collinear  for every $n\in V$. \\
					Plotted on the right are the graph $\mathcal{G}_1=(V, E_1)$, constructed in Example \ref{rem2} (a),  with vertices marked by blue dots and edges marked in black lines. Let $\mathbf{f}=(f_n)_{n\in V}$ be a complex-valued graph signal on $\mathcal{G}_1=(V, E_1)$. Then graph $\mathcal{G}_{1,\mathbf{f}}$ is connected with vertices from $A$ to $F$ corresponding to  $\{\triangle 145, \triangle 245, \triangle 345, \triangle 645, \triangle 745, \triangle 845\}$, marked in red filled triangles, and with undirected edges $\{\hskip-.02in (A,B),\hskip-.02in (A,C), \hskip-.02in (A,D), \hskip-.02in (A,E), \hskip-.02in (A,F), \hskip-.02in (B,C), \hskip-.02in (B,D), \hskip-.02in (B,E), \hskip-.02in (B,F), \hskip-.02in (C,D),\allowbreak (C,E),(C,F),(D,E),(D,F),(E,F)\}$, marked in red solid lines, where  the edge $(k,l)\in E_1$ does not pass the origin in the complex plane with $k=4, l=5$.
                    %$f_k$ and $f_l$ are non-collinear with $k=4, l=5$.
				} \label{fig1}
			\end{figure}

		}
		
	\end{example}

	\section{Conjugate phase retrieval in the Paley-Wiener space}\label{cpr.pw.sec}

The classical Whittaker-Shannon-Kotel'nikov (WSK) sampling theorem states that any function  in the Paley-Wiener space can be exactly reconstructed from its samples taken at the Nyquist rate or higher. Specifically, 
\begin{align*}
	f(x)=\sum_{n\in \Z} f(\frac{n}{2B}) sinc(2Bx-n), \ \forall x\in \R,
\end{align*}
where $f\in PW_B^2$ in \eqref{PW.def}. 
This reconstruction formula remains valid under uniform shifts of the sampling lattice. That is, for any fixed 
$x_0\in \R$, we have the equivalent representation:
\begin{align}\label{pw.rec.for1}
	f(x)=\sum_{n\in \Z} f(x_0+\frac{n}{2B}) sinc(2B(x-x_0)-n), \ \forall x\in \R,
\end{align}
where $f\in PW_B^2$ in \eqref{PW.def}. For the reader's convenience, we provide the proof of \eqref{pw.rec.for1} in Lemma \ref{lem.WSK}. %{\color{red} this is the modified version. If you use WSK, then you should write the classical version, that is sampled at the integer.}
%	\begin{theorem}[WSK sampling theorem]
	%		Let $B>0$ and the Paley-Wiener space $PW_B^2$ is defined as in (\ref{PW.def}). For any , we have
	%			where $sinc(x)=\frac{sin(\pi x)}{\pi x}$ for all $x\in \R$.
	%	\end{theorem}
%	

In \cite{mc2004phase}, the authors studied the determination of the  complex-valued bandlimited signals, up to a unimodular constant and conjugation,  from the magnitudes of structured convolutions $|f|$ and $|f(\cdot+c)-f|$ for some constant $c>0$, see below.  
\begin{lemma}\cite{mc2004phase}\label{lem2}
	Let $f\in PW^2_B$ be as in \eqref{PW.def} and $0<c\le \frac{1}{2B}$. Then $f$ can be determined, up to a unimodular constant and conjugation, from its absolute magnitudes $ |f(x)|, x\in \R $ and relative magnitudes $ |f(x + c)-f(x)|, x\in \R $.
\end{lemma}

From above lemma, we know that $f\in PW_B^2$ can be determined, up to a unimodular constant and conjugation, from the structured magnitudes $|f|$ and $|f(\cdot+ c)- f(\cdot)|$ on the whole real line. Then one may wonder if the conjugate phase retrieval of the bandlimited signals is possible from the phaseless samples taken on a discrete set. Combine the WSK sampling theorem with the observation that  $|f|^2\in PW_{2B}^1\subset PW_{2B}^2$ and $|f(\cdot+ c)- f(\cdot)|^2 \in PW_{2B}^1\subset PW_{2B}^2$ for any constant $c\in \R$ with $f\in PW_B^2$, we immediately have the following theorem.

\begin{theorem}\label{thm2}
	Let $f\in PW_B^2$ be as in \eqref{PW.def} and $0<c\le \frac{1}{2B}$. Then $f$ can be determined, up to a unimodular constant and conjugation, from its absolute magnitudes $ |f(\frac{n}{4B})|$ and  relative magnitudes $ |f(\frac{n}{4B} + c)-f(\frac{n}{4B})|, n\in \Z $.
\end{theorem}
\begin{proof}
	If $f=0$, it is trivial. For a nonzero signal $f\in PW_{B}^2$, let 
	$g\in PW_B^2$ satisfy 
	\begin{align}\label{eq64}
		|f(\frac{n}{4B})| = |g(\frac{n}{4B})| \text{ and } |f(\frac{n}{4B} + c)-f(\frac{n}{4B})| = |g(\frac{n}{4B} + c)-g(\frac{n}{4B})|,  n\in \Z.
	\end{align}
	One may easily verify that $|f|^2\in PW_{2B}^1\subset PW_{2B}^2$ and $|f(\cdot+ c)- f(\cdot)|^2 \in PW_{2B}^1\subset PW_{2B}^2$. By the WSK sampling theorem, we know that $|f|$ and $|f(\cdot+ c)- f(\cdot)|$ can be determined from  $|f(\frac{n}{4B})|$ and  $ |f(\frac{n}{4B} + c)-f(\frac{n}{4B})|, n\in \Z$, respectively. Combine this with \eqref{eq64}, we have 
	%Without loss of generality, we assume that $f$ and $g$ are nonzero. Since $ \mathcal{F}(|f|^2)= \mathcal{F}f\ast \mathcal{F}\overline{f} $ and $ \text{supp }\mathcal{F}f \subset [-B,B] $, we have $\text{supp }\mathcal{F}(|f|^2) \subset [-2B,2B]$. So $|f|^2 \in PW_{2B}^1 \subset PW_{2B}^2$. Since $ \mathcal{F} (f(\cdot+c)-f(\cdot))(w) = \mathcal{F}(f(\cdot+c))(w) - \mathcal{F}f(w) = \mathcal{F}f(w)(e^{2\pi icw}-1) $ and $ \text{supp }\mathcal{F}f \subset [-B,B] $, we have $ \text{supp } (f(\cdot+c)-f(\cdot)) \subset [-B,B] $. Similarly, $|f(\cdot+c) - f(\cdot)|^2 \in PW_{2B}^1 \subset PW_{2B}^2$. Furthermore, by (\ref{eq64}) and WSK sampling theorem, we have 
	\begin{align*}
		|f(x)|=|g(x)|, |f(x+c)-f(x)|=|g(x+c)-g(x)|, x\in \R.
	\end{align*}
	By Lemma \ref{lem2}, we have
	\begin{align*}
		f=e^{i\alpha} g \text{ or } f=e^{i\alpha} \overline{g}, \text{ for some } \alpha \in \R,
	\end{align*}
	which completes the proof.
\end{proof}

For a complex-valued signal  $f\in PW_{B}^2$ in \eqref{PW.def}, Theorem \ref{thm2} shows that it can be determined, up to a unimodular constant and conjugation, from its phaseless samples  $ |f(\frac{n}{4B})|$ and  $ |f(\frac{n}{4B} + c)-f(\frac{n}{4B})|, n\in \Z $ with four times the Nyquist rate.   In Theorems \ref{thm.reference} and \ref{PW.adj.thm1}, we see that conjugate phase retrieval  is possible for the graph signal with the absolute magnitudes on the vertices and relative magnitudes on the edges. We then apply these theorems to the function in the Paley-Wiener space, assuming that the graph signal residing on the graph is the function values. First, we explore the graph structure with two reference points in Theorem \ref{thm.reference}, and show that conjugate phase retrieval is possible with the structured phaseless samples at three times the Nyquist rate.
%
% In the following theorems, we show that conjugate phase retrieval is possible for the signal $f\in PW_B^2$ with structured phaseless samples assuming there exist two sampling points at which the function values are non-collinear, where we say that two distinct complex numbers $a$ and $b$ are non-collinear if $a\overline{b} \ne b\overline{a}$.
%We remark that the proofs of the following theorems do not utilize the analytic properties of the signal, and a  reconstruction algorithm can be derived directly from the proof, see Algorithm \ref{alg1}. 
%    First, we show that the complex-valued signal $f\in PW^2_B$ can be determined, up to a unimodular constant and conjugation, from the absolute magnitude samples  $ |f(x_0+\frac{n}{2B})|, n\in \Z $ and relative magnitude samples  $  |f(x_0+\frac{q}{2B})-f(x_0+\frac{n}{2B})|, q\in \{k,l\}, n\in \Z $ for some $x_0\in \R$, provided that there exist $k\ne l\in \Z$ such that $f(x_0+\frac{k}{2B})$ and $f(x_0+\frac{l}{2B})$ are not collinear, that is, 
%     \begin{align}\label{eq100}
	%     f(x_0+\frac{k}{2B})\overline{f(x_0+\frac{l}{2B})}\ne f(x_0+\frac{l}{2B})\overline{f(x_0+\frac{k}{2B})}.
	%     \end{align}

\begin{theorem}\label{pw.2pts.thm1}
	Let  $f\in PW_B^2$ in \eqref{PW.def} be a complex-valued function that is not real-valued up to a unimodular constant and $x_0\in \R$ be so chosen that $f(x_0+\frac{k}{2B})$ and $f(x_0+\frac{l}{2B})$ are non-collinear for some $k\ne l\in \Z$, that is, 
	\begin{align*}
		f(x_0+\frac{k}{2B})\overline{f(x_0+\frac{l}{2B})}\ne f(x_0+\frac{l}{2B})\overline{f(x_0+\frac{k}{2B})}.
	\end{align*}
	%\begin{align}\label{eq100}
	%     f(\frac{k}{2B})\overline{f(\frac{l}{2B})}\ne f(\frac{l}{2B})\overline{f(\frac{k}{2B})}.
	%   \end{align}
Then $f$ can be determined, up to a unimodular constant and conjugation, from absolute phaseless samples $ |f(x_0+\frac{n}{2B})|, n\in \Z $ and relative phaseless samples $  |f(x_0+\frac{n}{2B})-f(x_0+\frac{k}{2B})|, |f(x_0+\frac{n}{2B})-f(x_0+\frac{l}{2B})|, n\in \Z  $.
\end{theorem}
%As $\{x_0+\frac{n}{2B}\}_{n\in \Z}$ is a set of sampling  for $PW_{B}^2$ for any constant $x_0\in \R$, cf. Lemma \ref{lem.WSK}, the above theorem can be directly obtained from Theorem \ref{complete.cpr.thm1}.

In the above theorem, a sampling scheme with three times the Nyquist rate has been proposed, and we show that conjugate phase retrieval is possible for the signal $f\in PW_B^2$ with structured phaseless samples. % assuming there exist two sampling points at which the function values are non-collinear. . 
Specifically, if there exist two reference sampling points such that their corresponding function values are non-collinear, then $f\in PW_B^2$ can be determined, up to a unimodular constant and conjugation,  from the phaseless samples on the sampling set and relative phaseless samples with respect to
%between
the reference points, cf. Theorem \ref{thm.reference}.  Next, we will explore the circulant graph structure in Theorem  \ref{PW.adj.thm1}, and show that the phaseless samples on the sampling set and relative phaseless samples with respect to the first-order neighbors and seconde-order neighbors are sufficient for the conjugate phase retrieval of complex-valued bandlimited signals.

\begin{theorem}\label{PW.adj.thm2}
Let  $f\in PW_B^2$ in  \eqref{PW.def} be any complex-valued function and $x_0\in \R$ be so chosen that  
\begin{align}\label{eq97}
	f(x_0+\frac{n+1}{2B})\overline{f(x_0+\frac{n}{2B})} \ne f(x_0+\frac{n}{2B}) \overline{f(x_0+\frac{n+1}{2B})}
\end{align}
for all $n\in \Z$. Then $f$ can be determined, up to a unimodular constant and conjugation, from absolute phaseless samples $ |f(x_0+\frac{n}{2B})|, n\in \Z $ and relative phaseless samples $  |f(x_0+\frac{n+1}{2B})-f(x_0+\frac{n}{2B})|, |f(x_0+\frac{n+2}{2B})-f(x_0+\frac{n}{2B})|, n\in \Z$.   
\end{theorem}
%As $\{x_0+\frac{n}{2B}\}_{n\in \Z}$ is a set of sampling for $PW_{B}^{2}$ for any constant $x_0\in \R$, the above theorem can be directly obtained from Theorem \ref{cpr.adj.thm}. {\color{red}We remark that the sampling density of phaseless samples in Theorems \ref{pw.2pts.thm1}, \ref{PW.adj.thm2} corresponds to three times the Nyquist rate.}

%In the above theorem, we provide an alternative sampling scheme to show that $f\in PW_B^2$ can be determined, up to a unimodular constant and conjugation, from phaseless samples combined with relative phaseless samples between neighboring sampling points, cf. Theorem \ref{PW.adj.thm1}.  

\begin{remark}\label{rem1}
 The condition \eqref{eq97} holds for almost every $x_0\in \R$. 
%The condition on $x_0\in \R$ satisfying equation \eqref{eq97} holds almost everywhere. 
In fact, fix $c>0$ and let $f\in PW_B^2$ be any nonzero signal in the Paley-Wiener space, then $f(\cdot+c)\overline{f(\cdot)} - f(\cdot) \overline{f(\cdot+c)}$ is the restriction of an entire function to the real axis, and its zeros are either $\R$ or a countable set.  Since $f$ is a nonzero function, we obtain that the zeros of $f(\cdot+c)\overline{f(\cdot)} - f(\cdot) \overline{f(\cdot+c)}$ are countable.
\end{remark}

\begin{remark}
Theorem \ref{PW.adj.thm2} has been discussed in \cite{lai2021conjugate} for  the case of structured convolution with
\begin{align*}
	V= \begin{bmatrix}
		1 & 0 & 0 & 1 & 1 & 0 \\
		0 & 1 & 0 & -1 & 0 & 1 \\
		0 & 0 & 1 & 0 & -1 & -1 \\
	\end{bmatrix}
\end{align*}
and $b=(0,\frac{1}{2B},\frac{1}{B})^T$. However, the proof in this paper is different as the one therein. Also, for the reconstruction, we use a different approach other than the Gerchberg-Saxton method used in \cite{lai2021conjugate}, see Algorithm \ref{alg1}.
\end{remark}

We remark that from above theorems, we show that a complex-valued signal in the Paley-Wiener space can be determined, up to a unimodular and conjugation, from structured phaseless samples. However, no explicit reconstruction algorithm has been proposed in this context. In Section~\ref{cpr.sis.sec}, we adopt a different approach to establish conjugate phase retrieval in shift-invariant spaces, which can also be extended to the Paley-Wiener space. Moreover, this approach naturally leads to a numerical reconstruction algorithm based on the proof, see Algorithm \ref{alg1}.

If $f\in PW_B^2$ in \eqref{PW.def} is real-valued or real-valued up to a rotation, then \eqref{eq97} can not be satisfied. The following proposition tells us that  a real-valued signal $f\in PW_B^2$ can be determined, up to a sign,  from absolute phaseless samples $ |f(x_0 + \frac{n}{2B})|, n\in \Z$ and relative phaseless samples $|f(x_0 + \frac{n+1}{2B}) - f(x_0 + \frac{n}{2B})|, n\in \Z$ if  $f(x_0 + \frac{n}{2B})\ne 0$ for all $n\in \Z$ and some $x_0\in \R$, with twice the Nyquist rate.

\begin{proposition}\label{cpr.real.PW.prop}
Let $f\in PW_B^2$ in \eqref{PW.def} be any nonzero real-valued funcition and $x_0\in \R$ be so chosen  that $f(x_0 + \frac{n}{2B})\ne 0$ for all $n\in \Z$. Then $f$ can be determined, up to a sign, from absolute phaseless samples $ |f(x_0 + \frac{n}{2B})|, n\in \Z$ and relative phaseless samples $|f(x_0 + \frac{n+1}{2B}) - f(x_0 + \frac{n}{2B})|, n\in \Z$. 
\end{proposition}

% As $\{x_0+\frac{n}{2B}\}_{n\in \Z}$ is a set of sampling for $PW_{B}^{2}$ for any constant $x_0\in \R$, 
The above proposition can directly obtained from Proposition \ref{cpr.real.prop}.

\section{Conjugate phase retrieval in shift-invariant spaces}\label{cpr.sis.sec}

In order to complement existing research on conjugate phase retrieval in shift-invariant spaces generated by Gaussian-type functions or the sinc function, both of which rely on the theory of entire functions of finite order, we consider a broader setting. Specifically, we study conjugate phase retrieval in shift-invariant spaces generated by a real-valued, continuous generator, without relying on any analytic properties of the function. This general framework allows us to go beyond the classical analytic-based methods and address a wider class of shift-invariant spaces.  
 %In Section \ref{cpr.pw.sec}, we investigate the conjugate phase retrieval problem for signals in the Paley-Wiener space, which is generated by shifts of the sinc function. Building on this, we further study the conjugate phase retrieval in the more general shift-invariant space $V^p(\phi)$ defined in \eqref{Vspace.def}, where $\phi$  is a real-valued generator. 

 Let $V^p(\phi)$ be the shift-invariant space generated by a real function $\phi$ in \eqref{Vspace.def}. For a function $f\in V^p(\phi)$, if the  available phaseless samples are located on the vertices and edges of the graphs proposed in Theorems \ref{thm.reference} and \ref{PW.adj.thm1}, the feasibility of conjugate phase retrieval of the function  $f$ can be inferred from the structural phaseless samples established in Section \ref{cpr.graph.sec}.  However, these results do not provide an explicit reconstruction procedure for recovering the signal, and the phase propagation is depending on the construction of the ${\mathcal G}_{\bf f}$. To address this, in Section \ref{cpr.sis.sec}, we introduce an alternative approach that not only establishes the theoretical foundation for conjugate phase retrieval in shift-invariant spaces but also leads to a concrete reconstruction algorithm, see Algorithm \ref{alg2}.  This approach is further extended to the Paley-Wiener space, where a corresponding reconstruction procedure is developed, see Algorithm \ref{alg1}.

	A key technical component of our analysis is a geometric condition ensuring that a complex vector in $\C^3$ can be uniquely determined, up to a unimodular constant and conjugation, from its magnitudes and relative magnitudes, see below.
%provided that at least one pair of components does not lie on the same line through the origin. %This condition plays a central role in the proofs of our  theorems.	

	\begin{proposition}\label{prop2}
	A complex-valued vector $x=(x_1, x_2, x_3)\in \C^3$ can be determined, up to a unimodular constant and conjugation, from absolute magnitudes $ |x_n|, n\in \{1,2,3\} $ and relative magnitudes $  |x_n-x_m|, (n,m)\in \{1,2,3\}^2 $.
	\end{proposition}
	\begin{proof}
	 If $x=0$, it is trivial. 
	 
	 For a nonzero vector $x$,
		let $ y=(y_1, y_2, y_3)\in \C^3 $ satisfy
		\begin{align}\label{eq66}
			|x_n|=|y_n|,  n\in \{1,2,3\}
		\end{align}
		and
		\begin{align*}
			|x_n-x_m| = |y_n-y_m|,  (n,m)\in  \{1,2,3\}^2.
		\end{align*}
		Then
		\begin{align}\label{eq68}
			\Re x_n\Re x_m + \Im x_n\Im x_m = \Re y_n\Re y_m + \Im y_n\Im y_m,  (n,m)\in  \{1,2,3\}^2.
		\end{align}
		
	We divide the proof into two parts (a) and (b).
		
		(a) If $x$ is real-valued up to a unimodular constant, we may, without loss of generality, assume $x_1\ne 0$ and further assume $x_1=y_1> 0$, otherwise replacing $x$ by $ e^{i\alpha_1}x$ and $y$ by $ e^{i\alpha_2}y$ respectively, where $\alpha_1, \alpha_2 \in \R$. Then $x\in \R^3$, and we can obtain $x_2=y_2$ and $x_3=y_3$ by \eqref{eq68}. Thus, $x=y$, which implies $x$ is determined, up to a unimodular constant.
		
		(b) If $x$ is not real-valued up to a unimodular constant, then there exist $k,l\in \{1,2,3\}$ such that $x_{k}$ and $x_{l}$ are non-collinear, that is $x_{k} \overline{ x_{l}} \ne x_{l} \overline{ x_{k}}$. Without loss of generality, assume $x_{1} \overline{ x_{2}} \ne x_{2} \overline{ x_{1}}$, then $x_1, x_2 \ne 0$ and
		\begin{align}\label{eq92}
			\Re x_1 \Im x_2 \ne \Re x_2 \Im x_1.
		\end{align}
		Without loss of generality, assume $x_1=y_1>0$, otherwise replacing $x$ by $ e^{i\alpha_1}x $ and $y$ by $ e^{i\alpha_2}y $ respectively, where $\alpha_1, \alpha_2\in \R$.
		By \eqref{eq68}, we have $\Re x_2 = \Re y_2$ and then $\Im x_2 = \tau_2 \Im y_2 \ne 0$ for some $\tau_2 =\pm 1 $ by \eqref{eq66} and \eqref{eq92}. By \eqref{eq68}, we have
		\begin{equation*}
			\left\{ \begin{array}{ll}
				
				( \Re x_{3} - \Re y_{3} )  \Re x_{2}
				+ ( \Im x_{3} - \tau_2 \Im y_{3} )  \Im x_{2} {}& = 0 \\
				( \Re x_{3} - \Re y_{3} )  \Re x_{1}
				+ ( \Im x_{3} - \tau_2 \Im y_{3} )  \Im x_{1} {}& = 0
			\end{array}\right.
		\end{equation*}
		Together with \eqref{eq92}, we have $\Re x_{3} = \Re y_{3}$ and $\Im x_{3} = \tau_2 \Im y_{3}$. Thus, we have
		\begin{align*}
			x=y \text{ or } x = \overline{y},
		\end{align*}
		which implies $x$ is determined, up to a unimodular constant and conjugation.
	\end{proof}

 Proposition \ref{prop2} aligns with Proposition \ref{complete.prop.jfa} with  $d=2$.   For the remainder of this discussion, we will primarily leverage part (b) in the proof of Proposition \ref{prop2}, under the condition that the two involved points are non-collinear (meaning they are not aligned with the origin).
 
	   %It introduces the specific assumption that the graph signals
        %two vectors 
        %at the endpoints of at least one edge of the complete graph of order $3$
        %subgraph 
        %must be non-collinear (meaning they do not lie on the same line passing through the origin). 

	%Next, we apply Theorems \ref{thm.reference} and \ref{PW.adj.thm1} for the conjugate phase retrieval of graph signals to the function in the shift-invariant space, assuming that the graph signal residing on the graph is the function values. 
	First, inspired by the the graph structure with two reference points in Theorem \ref{thm.reference},  we show that conjugate phase retrieval is possible
	for complex-valued signal $f\in V^p(\phi)$, from the absolute magnitude samples  $ |f(x_n)|, n\in \Z $ together with the relative magnitude samples  $|f(x_n)-f(x_q)|, q\in \{k,l\}, n\in \Z $, where  $X=\{x_n\}_{n\in \Z}$ is a set of sampling for $V^p(\phi)$ satisfying \eqref{sampling.def.eq1},  and $k\ne l\in \Z$ are so chosen  that  $f(x_k)$ and $f(x_l)$ are non-collinear, that is, 
	\begin{align}\label{colinear.cond}
		f(x_k)\overline{f(x_l)}\ne f(x_l)\overline{f(x_k)}. 
	\end{align}

	\begin{theorem}\label{complete.cpr.thm1}
		Let $f\in V^p(\phi)$ be a complex-valued function that is not real-valued up to a unimodular constant,  and $X=\{x_n:n\in \Z\}$ be a set of sampling for $V^p(\phi)$ satisfying \eqref{sampling.def.eq1}. Then $f$  can be determined, up to a unimodular constant and conjugation, from absolute magnitude samples $ |f(x_n)|, x_n\in X $ and relative magnitude samples $  |f(x_{n})-f(x_k)|, |f(x_{n})-f(x_l)|, x_n\in X$, where  $k\ne l$ are so chosen that \eqref{colinear.cond} holds.
	\end{theorem}
	\begin{proof}
		Let $g\in V^p(\phi)$ satisfy 
		\begin{align}\label{eq93}
			|f(x_n)| = |g(x_n)|,
		\end{align}
		\begin{align*}
			|f(x_{n}) - f(x_q)| = |g(x_{n}) - g(x_q)|, q\in \{k,l\}
		\end{align*}
		for all $x_n\in X$.
		Then
		\begin{align}\label{eq96}
			\Re f(x_{n}) \Re f(x_{q}) + \Im f(x_{n}) \Im f(x_{q}) = \Re g(x_{n}) \Re g(x_{q}) + \Im g(x_{n}) \Im g(x_{q})
		\end{align}
		for all $x_n\in X$ and $q\in \{k,l\}$.
		By \eqref{colinear.cond}, we have
		\begin{align}\label{eq103}
			\Re f(x_{k}) \Im f(x_{l}) \ne \Re f(x_{l}) \Im f(x_{k}).
		\end{align}
		Obviously, $f$ is nonzero. In fact, we observe $f(x_{k})\ne 0$ and $f(x_{l})\ne 0$ from \eqref{eq103}.
		Without loss of generality, assume $f(x_{k})= g(x_{k}) >0$, otherwise replacing $f$ by $ e^{i\alpha_1}f $ and $g$ by $ e^{i\alpha_2}g $ respectively, where $\alpha_1, \alpha_2\in \R$.
		
		By \eqref{eq96}, we have $\Re f(x_{l}) = \Re g(x_{l})$. Combine this with \eqref{eq93} and \eqref{eq103}, we have
		\begin{align}\label{eq104}
			\Im f(x_{l}) = \tau_l \Im g(x_{l}) \ne 0 \text{ for some } \tau_l = \pm 1.
		\end{align}
		Hence $(f(x_{k}), f(x_{l}))=(g(x_{k}), g(x_{l}))$ or $(f(x_{k}), f(x_{l}))=\overline{(g(x_{k}), g(x_{l}))}$. Following the  similar  arguments in part (b) of the proof of Proposition \ref{prop2}, we have 
		$$\Re (f(x_{n}), f(x_{k}), f(x_{l}) ) = \Re (g(x_{n}), g(x_{k}), g(x_{l}))$$ and $$\Im (f(x_{n}), f(x_{k}), f(x_{l}))= \tau_n \Im (g(x_{n}), g(x_{k}), g(x_{l}))$$ for all $n\in \Z \setminus \{k,l\}$, where $\tau_n=\pm 1$.
        Combine this with \eqref{eq104}, we have $\tau_n=\tau_l$ for all $n\in \Z \setminus \{k,l\}$. Therefore 
		\begin{align*}
			f(x_n)=g(x_n) \text{ for all } x_n\in X \ \ \text{ or }  \ \ f(x_n)=\overline{g(x_n)} \text{ for all } x_n\in X.
		\end{align*}
		As $X$ is a set of sampling for $V^p(\phi)$, we obtain
		\begin{align*}
			f=g \text{ or } f=\overline{g},
		\end{align*}
		which completes  the proof.
	\end{proof}
	
	%The proof of the above theorem can be adapted from the  argument used in the proof of Theorem \ref{complete.cpr.thm1} for the case (a). 	
	
	%In Theorem \ref{complete.cpr.thm1}, to determine the signal $f\in V^p(\phi)$, up to a unimodular constant and conjugation, we need $\sharp(X)^2+\sharp(X)$ structured samples. It is natural to consider if less samples can ensure the conjugate phase retrieval  for signals $f\in V^p(\phi)$.  %In the next theorem, we show that conjugate phase retrieval is possible for the signal $f\in V^p(\phi)$ from absolute magnitudes $ |f(x_n)|, n\in \Z $ and relative magnitudes $  |f(x_{n+1})-f(x_n)|, n\in \Z $ and $ |f(x_{n+2})-f(x_n)|, n\in \Z$, where $X$ is a set of sampling satisfies \eqref{sampling.def.eq1}. }
	Next, we will explore the circulant graph structure in Theorem \ref{PW.adj.thm1}, and show that the absolute magnitudes and relative magnitudes with respect to the first-order
	neighbors and seconde-order neighbors are sufficient for the conjugate phase retrieval of complex-valued signal $f\in V^p(\phi)$. Specially, in the next theorem, we show that $f\in V^p(\phi)$ can be determined, up to a unimodular constant and conjugation, from absolute magnitude samples $|f(x_n)|,x_n\in X$, and relative magnitude samples $|f(x_{n+1})-f(x_n)|, |f(x_{n+2})-f(x_n)|, x_n\in X$, provided that $X=\{x_n\}_{n\in \Z}$ is a set of sampling for $V^p(\phi)$ satisfying \eqref{sampling.def.eq1} and the function values at adjacent sampling points are non-collinear, that is,
	\begin{align}\label{eq91}
		f(x_{n+1})\overline{f(x_n)} \ne f(x_n) \overline{f(x_{n+1})}
	\end{align}
	for all $x_n\in X$.

\begin{theorem}\label{cpr.adj.thm}
	Let $f\in V^p(\phi)$ be a complex-valued function,  $X=\{x_n\}_{n\in \Z}$ be a set of sampling for $V^p(\phi)$ satisfying \eqref{sampling.def.eq1}  and \eqref{eq91}. Then $f$ can be determined, up to a unimodular constant and conjugation, from absolute magnitude samples $ |f(x_n)|, n\in \Z $ and relative magnitude samples $|f(x_{n+1})-f(x_n)|, |f(x_{n+2})-f(x_n)|, n\in \Z$.
\end{theorem}
\begin{proof}
	Let $g\in V^p(\phi)$ satisfy 
	\begin{align}\label{eq24}
		|f(x_n)| = |g(x_n)|, 
	\end{align}
	\begin{align*}
		|f(x_{n+1}) - f(x_n)| = |g(x_{n+1}) - g(x_n)|
	\end{align*}
	and
	\begin{align*}
		|f(x_{n+2}) - f(x_n)| = |g(x_{n+2}) - g(x_n)|
	\end{align*}
	for all $n\in \Z$.  
	Then
	\begin{align}\label{eq28}
		\Re f(x_{n+1})   \Re f(x_{n})    +  \Im f(x_{n+1})  \Im f(x_{n}) 
		=  \Re g(x_{n+1})   \Re g(x_{n}  )  +  \Im g(x_{n+1})   \Im g(x_{n}) 
	\end{align}
	and
	\begin{align}\label{eq29}
		\Re f(x_{n+2})   \Re f(x_{n})    +  \Im f(x_{n+2})  \Im f(x_{n}) 
		=  \Re g(x_{n+2})   \Re g(x_{n}  )  +  \Im g(x_{n+2})   \Im g(x_{n}) 
	\end{align}
	for all $n\in \Z$.
	By \eqref{eq91}, we have
	\begin{align}\label{eq73}
		\Re f(x_{n+1}) \Im f(x_{n}) \ne \Re f(x_{n}) \Im f(x_{n+1})
	\end{align}
	for all $n\in \Z$.
	Then we know that $f(x_n) \ne 0$ for all $n\in \Z$. Without loss of generality, we assume $f(x_0)=g(x_0)>0$, otherwise replacing $f$ by $e^{i\alpha_1}f$ and $g$ by $e^{i\alpha_2}g$ respectively, where $\alpha_1, \alpha_2\in \R$.
	Obviously, $\Re f(x_0) = \Re g(x_0) > 0$ and $ \Im f(x_0) = \Im g(x_0) =0$. By \eqref{eq24}, \eqref{eq28} and \eqref{eq73}, we have $\Re f(x_1) =  \Re g(x_1)$ and $ \Im f(x_1)= \tau_1 \Im g(x_1) \ne 0$ for some $\tau_1= \pm 1$. Following the similar arguments in part (b) of the proof of Proposition \ref{prop2}, we have
	\begin{align*}
		\Re\big(f(x_{0}), f(x_{1}), f(x_{2})\big) = \Re\big (g(x_{0}), g(x_{1}), g(x_{2})\big)
	\end{align*}
	and
	\begin{align*}
		\Im \big(f(x_{0}), f(x_{1}), f(x_{2}))= \tau_1 \Im (g(x_{0}), g(x_{1}), g(x_{2})\big).
	\end{align*}
	Thus, $(f(x_0), f(x_1), f(x_2))$ is determined, up to a conjugation.
	
	Next we will prove that $(f(x_n))_{n\in \Z}$ is determined, up to a conjugation, 
	%$, i.e.,
	%\begin{align}\label{eq36}
	%	\Re(f(x_n)) = \Re g(x_n), \Im (f(x_n)) = \tau_1 \Im  g(x_n),  n\in \Z
	%\end{align}
	by induction.
	Inductively, we assume
	\begin{align*}
		\Re f(x_n) = \Re g(x_n), \Im f(x_n) = \tau_1 \Im  g(x_n),  n\in [0,n_1]\cap \Z,
	\end{align*}
	where $n_1\ge 2$.
	By \eqref{eq28} and \eqref{eq29},  we have
	\begin{equation*}
		\left\{ \hskip .2in \begin{array}{ll}
			\hskip-.2in
			\big( \Re f(  x_{n_1+1}  )  -  \Re g(  x_{n_1+1}  ) \big)  \Re f(x_{n_1}) + \big(  \Im f(  x_{n_1+1}  )  -  \tau_1 \Im g(  x_{n_1+1}) \big)  \Im f(x_{n_1}  ) {}& \hskip-.10in =  0 \\
			\hskip-.2in
			\big( \Re f(  x_{n_1+1}  )  -  \Re g(  x_{n_1+1}  ) \big)  \Re f(x_{n_1-1}  ) 
			+   \big( \Im  f(  x_{n_1+1}  )  -  \tau_1 \Im g(  x_{n_1+1}  )  \big)  \Im  f(  x_{n_1-1}  ) {}& \hskip-.10in =   0	
		\end{array}\right.
	\end{equation*}
	Together with (\ref{eq73}), we have $\Re f(x_{n_1+1}) = \Re g(x_{n_1+1}) $ and $\Im  f(x_{n_1+1}) = \tau_1 \Im g(x_{n_1+1})$. 
	Therefore, we have 
	\begin{align}\label{cpr.eq.1a}
		\Re f(x_n) = \Re g(x_n) \ \ {\rm and} \ \ \Im f(x_n) = \tau_1 \Im  g(x_n)
	\end{align}
	for all $n\in [0,\infty)\cap \Z$. Following the similar arguments above, we  have \eqref{cpr.eq.1a}  holds for all $n\in (-\infty, 0)\cap \Z$. Therefore 
	\begin{align*}
		f(x_n)=g(x_n) \text{ for all } x_n\in X  \ \ \text{ or }  \ \ f(x_n)=\overline{g(x_n)} \text{ for all } x_n\in X.
	\end{align*}
	As $X$ is a set of sampling for $V^p(\phi)$, we obtain
	\begin{align*}
		f=g \text{ or } f=\overline{g},
	\end{align*}
	which completes the proof.
\end{proof}

If $f\in V^p(\phi)$ in \eqref{Vspace.def} is real-valued or real-valued up to a rotation, then \eqref{eq91} can never hold. The following proposition tells us that  a real-valued signal $f\in V^p(\phi)$ can be determined, up to a sign,  from absolute magnitude samples $ |f(x_n)|, x_n\in X$ and relative magnitude samples $|f(x_{n+1})-f(x_n)|, x_n\in X$, provided that  $X=\{x_n\}_{n\in \Z} \subset \R$ is a set of sampling for $V^p(\phi)$ satisfying \eqref{sampling.def.eq1} and   $f(x_n)\ne 0$ for all $x_n\in X$.

\begin{proposition}\label{cpr.real.prop}
	Let $f\in V^p(\phi)$ be any nonzero real-valued function, $X=\{x_n\}_{n\in \Z} \subset \R$ be a set of sampling for $V^p(\phi)$ satisfying \eqref{sampling.def.eq1} and   $f(x_n)\ne 0$ for all $x_n\in X$. Then $f$ can be determined, up to a sign, from absolute magnitude samples $ |f(x_n)|, x_n\in X$ and relative magnitude samples $|f(x_{n+1})-f(x_n)|, x_n\in X$. 
\end{proposition}
\begin{proof}
	Let $g\in V^p(\phi)$ satisfy \begin{equation}\label{p.eq1}
		|f(x_n)|=|g(x_n)|
	\end{equation}
	and
	\begin{equation}\label{p.eq2}
		|f(x_{n+1})-f(x_n)|=|g(x_{n+1})-g(x_n)|
	\end{equation}
	for all $x_n\in X$. Without loss of generality, we assume $f(x_0)=g(x_0)$, otherwise replacing $f$ by $-f$. 
	As $f(x_n)\neq 0, n\in \Z$, we have $f(x_{1})=g(x_{1})$ by \eqref{p.eq1} and \eqref{p.eq2}. Inductively, we assume that 
	$(f(x_0), f(x_{1}), \ldots, f(x_{n_1}))=(g(x_0), g(x_{1}), \ldots, g(x_{n_1}))$ for some $n_1\in [1,\infty)\cap \Z$. By \eqref{p.eq1} and \eqref{p.eq2}, we have $f(x_{n_1+1})=g(x_{n_1+1})$. 
	By induction, we have $f(x_n)=g(x_n)$ for all $n\in [0,\infty)\cap \Z$. Following the similar arguments above, we have $f(x_n)=g(x_n)$ for all $n\in (-\infty,0)\cap \Z$. Therefore, $f(x_n)=g(x_n)$ for all $x_n\in X$. As $X$ is a set of sampling for $V^p(\phi)$ satisfying \eqref{sampling.def.eq1}, we have 
	\begin{align*}
		f=g,
	\end{align*}
	which completes the proof.
\end{proof}

To conclude this section, we present a reconstruction algorithm inspired by the proof of Theorem \ref{complete.cpr.thm1}, see Algorithm \ref{alg2}. In addition, we provide a reconstruction algorithm for conjugate phase retrieval of functions in the Paley-Wiener space, given in Algorithm \ref{alg1}. We emphasize that the design of Algorithm \ref{alg1} is based on the proof of Theorem \ref{cpr.adj.thm}, which does not rely on the analytic properties of the underlying function.

\begin{algorithm}[!ht]
	\renewcommand{\algorithmicrequire}{\textbf{Input:}}
	\renewcommand{\algorithmicensure}{\textbf{Output:}}
	\caption{Reconstruct $\{f(x_n)\}_{n\in \Z}$ in $V^p(\phi)$ from structured phaseless samples, up to a unimodular constant and conjugation}\label{alg2}
	
	\begin{algorithmic}[1]
		\State{\bf Inputs}: Suppose $\{x_n\}_{n\in \Z}$ is a set of sampling for $V^p(\phi)$ satisfying \eqref{sampling.def.eq1} and $k\ne l\in \Z$ satisfies \eqref{colinear.cond}; Write $f(x_n)$ by $f_n, n\in \Z$; Phaseless samples $|f_n| \cup |f_{n} - f_k| \cup |f_{n} - f_l|$ for all $n\in \Z$.
		\State {\bf Initials}: Choose $g_k = |f_k|$, $ g_l= \frac{|f_l-f_k|^2 - |f_l|^2 - |f_k|^2}{-2{ |f_k|}} + i \sqrt{|f_l|^2 - (\frac{|f_l-f_k|^2 - |f_l|^2 - |f_k|^2}{-2{ |f_k|}})^2} $, up to a unimodular constant and conjugation.
		
		\State \textbf{For} $n\ne k,l$ \textbf{do}
		
		Determine $g_{n}$: 
		\begin{align*}
			\begin{pmatrix}
				\Re g_{n} \\
				\Im g_{n}
			\end{pmatrix}
			= \frac{1}{-2}  \begin{pmatrix}
				0  & 1/g_k \\
				1/\Im g_l & - \Re g_l / (g_k \Im g_l)
			\end{pmatrix}
			\begin{pmatrix}
				{|f_{n} - f_{l}|^2 - |f_{n}|^2 - |f_{l}|^2} \\
				{|f_{n} - f_{k}|^2 - |f_{n}|^2 - |f_{k}|^2}
			\end{pmatrix}.
		\end{align*} 
		%(\Re g_l & \Im g_l \\ \Re g_k & \Im g_k )
		\textbf{end}
		\State {\bf Output}: $f_n=g_n$ for all $n\in \Z$.
	\end{algorithmic}
\end{algorithm}

\begin{algorithm}[!ht]
\renewcommand{\algorithmicrequire}{\textbf{Input:}}
\renewcommand{\algorithmicensure}{\textbf{Output:}}
\caption{Reconstruct $f$ in $PW_B^2$ from structured phaseless samples, up to a unimodular constant and conjugation}\label{alg1}
\begin{algorithmic}[1]
	\State{\bf Inputs}: Suppose $x_0$ satisfies \eqref{eq97}; Write $f(x_0+\frac{n}{2B})$ by $f_n, n\in \Z$; Phaseless samples $|f_n| \cup |f_{n+1} - f_n| \cup |f_{n+2} - f_n|$ for all $n\in \Z$.
	\State {\bf Initials}: Choose $g_0 = |f_0|$, $ g_1= \frac{|f_1-f_0|^2 - |f_1|^2 - |f_0|^2}{-2{|f_0|}} + i \sqrt{|f_1|^2 - (\frac{|f_1-f_0|^2 - |f_1|^2 - |f_0|^2}{-2{|f_0|}})^2}$, up to a unimodular constant and conjugation.
	
	\State \textbf{For} $n\ge 2$ \textbf{do}
	
	Determine $g_{n}$: 
	\begin{align*}
		\begin{pmatrix}
			\Re g_{n} \\
			\Im g_{n}
		\end{pmatrix}
		= \frac{1}{-2}  \begin{pmatrix}
			\Re g_{n-1} & \Im g_{n-1} \\
			\Re g_{n-2} & \Im g_{n-2}
		\end{pmatrix}^{-1}
		\begin{pmatrix}
			{|f_{n} - f_{n-1}|^2 - |f_{n}|^2 - |f_{n-1}|^2} \\
			{|f_{n} - f_{n-2}|^2 - |f_{n}|^2 - |f_{n-2}|^2}
		\end{pmatrix}.
	\end{align*}
	\textbf{end}
	\State \textbf{For} $n\le -1$ \textbf{do}
	
	Determine $g_{n}$:
	\begin{align*}
		\begin{pmatrix}
			\Re g_{n} \\
			\Im g_{n}
		\end{pmatrix} 
		= \frac{1}{-2} \begin{pmatrix}
			\Re g_{n+1} & \Im g_{n+1} \\
			\Re g_{n+2} & \Im g_{n+2}
		\end{pmatrix}^{-1}
		\begin{pmatrix}
			{|f_{n} - f_{n+1}|^2 - |f_{n}|^2 - |f_{n+1}|^2} \\
			{|f_{n} - f_{n+2}|^2 - |f_{n}|^2 - |f_{n+2}|^2}
		\end{pmatrix}.
	\end{align*}
	\textbf{end}
	\State {\bf Output}: $f=\sum_{n\in \Z} g_n {\rm sinc}(2B(\cdot - x_0) - n)$.
\end{algorithmic}

\end{algorithm}

\section{STFT conjugate phase retrieval in complex Paley-Wiener space}\label{stft.pw.sec}

The \textit{short-time Fourier transform (STFT)}  is a fundamental tool in time-frequency analysis for a function $f\in L^2(\R)$ with  a given window function  $\psi\in L^2(\R)$, defined as
\begin{align*}
	\mathcal{V}_\psi f(x,w):= \int_{\R} f(t)\overline{\psi(t-x)}e^{-2\pi itw}dt, \  x,w\in \R. 
\end{align*} 
Our interest lies in recovering a function from STFT magnitudes sampled only in the time domain, that is the frequency component $\omega=0$. In this case, the STFT simplifies to the convolution of the function $f$ with the reflected conjugate of the window, that is, 
 $$ \mathcal{V}_{\psi} f(x,0) = \int_{\R} f(t)\overline{\psi(t-x)} dt= f\ast \psi^\# (x)$$ 
where $\psi^\#(x) = \overline{\psi(-x)}$. 
From measurements of the form $\{|V_{\psi}f(x, 0)|: x\in X\}$ with a real-valued window function $\psi$, $e^{i\alpha_1} f$ and $e^{i\alpha_2} \overline{f}$ are not distinguishable, where $f$ is the complex-valued bandlimited function,  $X$ is a discrete sampling set, and $\alpha_1, \alpha_2\in \R$.   Analogous  to the conjugate phase retrieval of bandlimited function in Section \ref{cpr.pw.sec}, we explore the   conjugate phase retrieval of complex-valued bandlimited functions from   the structured  STFT magnitudes only measured in the time domain.

%{\color{red}The \textit{short-time Fourier transform (STFT)} is widely used in time-frequency analysis \cite{alaifari2021uniqueness,grochenig2001foundations}.
%Define the \textit{STFT} of $f\in L^2(\R)$ by
%\begin{align*}
%	\mathcal{V}_\psi f(x,w):= \int_{\R} f(t)\overline{\psi(t-x)}e^{-2\pi itw}dt, x,w\in \R,
%\end{align*}
%with the window function $\psi\in L^2(\R)$. One may easily verify  that $\mathcal{V}_\psi f$ is uniformly continuous \cite{grochenig2001foundations}. The authors in \cite{alaifari2021uniqueness} studied STFT phase retrieval for real- and complex-valued signals. However, for a complex-valued signal, $e^{i\alpha_1} f$ and $e^{i\alpha_2} \overline{f}$ cannot be distinguished from phaseless STFT samples only in the time domain. This is again a conjugate phase retrieval problem.
%In the context of this paper, we leverage conjugate phase retrieval framework of bandlimited signals from phaseless point evaluations to establish conjugate phase retrieval for bandlimited signals from STFT magnitudes only in the time domain.}

%We first give the well-known Ambiguity function relation below, see \cite{alaifari2021uniqueness} for the proof.

%\begin{lemma}[Ambiguity function relation]\label{lem1}
%	Let $f,\psi\in L^2(\R)$. Then
%	\begin{align*}
%		\mathcal{F}(|\mathcal{V}_\psi f|^2)(w,-x) = \mathcal{V}_f f(x,w) \mathcal{V}_\psi \psi(x,w), \  (x,w)\in \R^2.
%	\end{align*}
%\end{lemma}

In \cite{alaifari2021uniqueness}, the authors proved that a real-valued function in $PW_B^2$ can be determined, up to a sign, from its phaseless STFT samples in the time domain using one window function, with a sampling rate twice the Nyquist rate.
In the following theorem, we show that a complex-valued function  in $PW_B^2$ can be determined, up to a unimodular constant and conjugation, from its phaseless STFT samples in the time domain using two structured window functions, with a sampling rate four times the Nyquist rate, which is analogous to Theorem \ref{thm2}.

%with  windows $\psi$ and $\psi(\cdot-c)-\psi(\cdot),  0<c\le \frac{1}{2B}$, with a sampling rate four times the Nyquist rate.

\begin{theorem}\label{thm4}
	Let $B>0, 0<c\le \frac{1}{2B}$, the Paley-Wiener space $PW_B^2$ be defined as in \eqref{PW.def} and $\psi\in L^2(\R)$ be a real-valued window such that
	\begin{equation}\label{eq03}
		(\mathcal{F} \psi) (w) \ne 0 ,  \ \ a.e. \ \  w \in (-B,B).
	\end{equation}
	Then any complex-valued function $f\in PW_B^2$ can be determined, up to a unimodular constant and conjugation, from $ |\mathcal{V}_{\psi} f (x_0+\frac{n}{4B},0)|, n\in \Z $ and $ |\mathcal{V}_{\psi(\cdot-c)-\psi(\cdot)} f (x_0+\frac{n}{4B},0)|, n\in \Z  $, where $x_0\in \R$ is any constant.
\end{theorem}
\begin{proof}
	Note that $\mathcal{V}_{\psi} f(x,0) = f\ast \psi^\# (x)$ and $\mathcal{V}_{\psi(\cdot-c)-\psi(\cdot)} f(x,0) = f\ast \psi^\# (x+c) - f\ast \psi^\# (x)$ for any $x,c\in \R$.
	Let $g\in PW_B^2$ satisfy
	
	\begin{align}\label{eq01}
		|f\ast \psi^\# (x_0 + \frac{n}{4B})| = |g\ast \psi^\# (x_0+\frac{n}{4B})|
	\end{align}
	and
	\begin{align}\label{eq02}
		|f\ast \psi^\# (x_0 \hskip-.03in + \hskip-.03in \frac{n}{4B} \hskip-.03in + \hskip-.03in c) \hskip-.03in - \hskip-.03in f\ast \psi^\# (x_0 \hskip-.03in + \hskip-.03in \frac{n}{4B})| \hskip-.03in = \hskip-.03in |g\ast \psi^\# (x_0 \hskip-.03in + \hskip-.03in \frac{n}{4B} \hskip-.03in + \hskip-.03in c) \hskip-.03in - \hskip-.03in g\ast \psi^\# (x_0 \hskip-.03in + \hskip-.03in \frac{n}{4B})|
	\end{align}
	for all $n\in \Z$.
	If $f$ is a  zero signal, it is trivial. Assume $f$ is a nonzero signal. 
	%By Lemma \ref{lem1}, we have
	%\begin{align}\label{eq1}
	%	\mathcal{F}(|f\ast \psi^\#|^2)(\cdot) =  \mathcal{F}(|\mathcal{V}_{\psi} f|^2) (\cdot,0) = \mathcal{V}_f f (0,\cdot) \overline{ \mathcal{V}_\psi \psi (0,\cdot) } .
	%\end{align}
	%Since $f\in PW_B^2$, we have
	%\begin{align}\label{eq101}
	%	\text{supp } \mathcal{F}(|f\ast \psi^\#|^2) \subset \text{supp } \mathcal{V}_f f (0,\cdot) = \text{supp } \mathcal{F}(|f|^2) \subset [-2B,2B].
	%\end{align}
	
	%Since $f,\psi\in L^2(\R)$, we have $\mathcal{F}f, \mathcal{F}\psi^\# \in L^2(\R)$. 

	By the convolution theorem, we have $\mathcal{F}(f\ast \psi^\#) = (\mathcal{F}f) (\mathcal{F}\psi^\#) \in L^1([-B,B])$, which implies $f\ast \psi^\# \in PW_B^1 \subset PW_B^2$.  Combine this with \eqref{eq01}, \eqref{eq02} and Theorem \ref{thm2}, we have
	%Then $f\ast \psi^\# \in L^2(\R)$ and thus $| f\ast \psi^\# |^2 \in L^1(\R)$. Combining \eqref{eq101} implies $|f\ast \psi^\#|^2 \in PW_{2B}^1 \subset PW_{2B}^2$.
	\begin{align*}
		f\ast \psi^\# = e^{i\alpha} g\ast \psi^\# \text{ or } f\ast \psi^\# = e^{i\alpha} \overline{g\ast \psi^\#}, \text{ for some } \alpha\in \R.
	\end{align*}

	If $f\ast \psi^\# = e^{i\alpha} g\ast \psi^\#$ for some $\alpha\in \R$, then $\mathcal{F}f \mathcal{F}\psi^\# = e^{i\alpha} \mathcal{F}g \mathcal{F}\psi^\#$. Since $\mathcal{F} \psi^\# (w) = \mathcal{F}\psi (-w)$, then it follows from \eqref{eq03} that $\mathcal{F}f(w) = e^{i\alpha} \mathcal{F}g(w), a.e. w\in \R$ and thus $f=e^{i\alpha} g$. 
	%Since a bandlimited signal is the restriction of an entire function to the real line, by the Identity Theorem for analytic functions, we have $f= e^{i\alpha} g $.
	If $f\ast \psi^\# = e^{i\alpha} \overline{g\ast \psi^\#}$ for some $\alpha\in \R$. Similarly, we can obtain $f= e^{i\alpha} \overline{g}$ for some $\alpha\in \R$. Then we complete the proof.
\end{proof}

We observe that the result in Theorem \ref{thm4} is merely a theoretical sampling result and does not correspond to a reconstruction method. Next  we 
explore the circulant graph structure in Theorem  \ref{PW.adj.thm1} and
propose a sampling scheme analogous to Theorem \ref{PW.adj.thm2} to achieve STFT conjugate phase retrieval in $PW_B^2$ using three structured window functions, which corresponds to a reconstruction method, with a sampling rate three times the Nyquist rate.
%using three  windows $\psi$, $\psi(\cdot-\frac{1}{B})-\psi(\cdot)$ and $\psi(\cdot-\frac{1}{2B})-\psi(\cdot)$, which corresponds to a reconstruction method, with a sampling rate three times the Nyquist rate.

\begin{theorem}\label{thm6}
	Let $B>0$, the Paley-Wiener space $PW_B^2$ be defined as in (\ref{PW.def}), $\psi\in L^2(\R)$ be a real-valued window function such that
	\begin{align*}
		(\mathcal{F} \psi) (w) \ne 0 , \  \ a.e. \  \ w \in (-B,B),
	\end{align*}
	and let $f$ be any nonzero complex-valued function in $PW_B^2$ and $x_0\in \R$ satisfy
	\begin{align}\label{eq50}
		\mathcal{V}_{\psi} f (x_0 \hskip-.03in + \hskip-.03in \frac{n+1}{2B},0) \overline{\mathcal{V}_{\psi} f (x_0 \hskip-.03in + \hskip-.03in \frac{n}{2B},0)} \ne \mathcal{V}_{\psi} f (x_0 \hskip-.03in + \hskip-.03in \frac{n}{2B},0) \overline{\mathcal{V}_{\psi} f (x_0 \hskip-.03in + \hskip-.03in \frac{n+1}{2B},0)} ,  n\in \Z.
	\end{align}
	Then $f$ can be determined, up to a unimodular constant and conjugation, from magnitudes  $ |\mathcal{V}_{\psi} f (x_0+\frac{n}{2B},0)|$, $|\mathcal{V}_{\psi(\cdot-\frac{1}{2B})-\psi(\cdot)} f (x_0+\frac{n}{2B},0)|$ and $|\mathcal{V}_{\psi(\cdot-\frac{1}{B})-\psi(\cdot)} f (x_0+\frac{n}{2B},0)|, n\in \Z$.
\end{theorem}
\begin{proof}
	Note that $\mathcal{V}_{\psi} f(x,0) = f\ast \psi^\# (x)$ and $\mathcal{V}_{\psi(\cdot-c)-\psi(\cdot)} f(x,0) = f\ast \psi^\# (x+c) - f\ast \psi^\# (x)$ for any $x,c\in \R$.
	Let $g\in PW_B^2$ satisfy
	\begin{align*}
		|f\ast \psi^\# (x_0+\frac{n}{2B})| = |g\ast \psi^\# (x_0+\frac{n}{2B})|,
	\end{align*}
	\begin{align*}
		|f\ast \psi^\# (x_0+\frac{n+1}{2B}) - f\ast \psi^\# (x_0+\frac{n}{2B})| = |g\ast \psi^\# (x_0+\frac{n+1}{2B}) - g\ast \psi^\# (x_0+\frac{n}{2B})|
	\end{align*}
	and
	\begin{align*}
		|f\ast \psi^\# (x_0+\frac{n+2}{2B}) - f\ast \psi^\# (x_0+\frac{n}{2B})| = |g\ast \psi^\# (x_0+\frac{n+2}{2B}) - g\ast \psi^\# (x_0+\frac{n}{2B})|
	\end{align*}
	for all $n\in \Z$.
	By (\ref{eq50}), we have
	\begin{align*}
		f\ast \psi^\# (x_0 \hskip-.03in + \hskip-.03in \frac{n+1}{2B}) \overline{f\ast \psi^\# (x_0 \hskip-.03in + \hskip-.03in \frac{n}{2B})} \hskip-.03in \ne \hskip-.03in f\ast \psi^\# (x_0 \hskip-.03in + \hskip-.03in \frac{n}{2B}) \overline{f\ast \psi^\# (x_0 \hskip-.03in + \hskip-.03in \frac{n+1}{2B})} , n\in \Z.
	\end{align*}
	By the proof of Theorem \ref{thm4}, we have $f\ast \psi^\#, g\ast \psi^\# \in PW_B^2$, and then
	\begin{align*}
		f\ast \psi^\# = e^{i\alpha} g\ast \psi^\# \text{ or } f\ast \psi^\# = e^{i\alpha} \overline{g\ast \psi^\#}, \text{ for some } \alpha\in \R, 
	\end{align*}
	by Theorem \ref{PW.adj.thm2}. 
	Applying the arguments in the proof of Theorem \ref{thm4}, we obtain the conclusion.
\end{proof}

\begin{remark}
	We remark that Theorem \ref{thm6} implies a reconstruction method similar to Algorithm \ref{alg1}, using magnitudes $ |\mathcal{V}_{\psi} f (x_0+\frac{n}{2B},0)|$, $|\mathcal{V}_{\psi(\cdot-\frac{1}{2B})-\psi(\cdot)} f (x_0+\frac{n}{2B},0)|$ and $|\mathcal{V}_{\psi(\cdot-\frac{1}{B})-\psi(\cdot)} f (x_0+\frac{n}{2B},0)|, n\in \Z$.
\end{remark}

\begin{appendices}
	\section{}

	%%--------- the following is less important -----------%%

	%% ----- Did not revise yet -------------------

	In Algorithm \ref{alg1}, we reconstruct the signal $f\in PW_{B}^2$ from the structured phaseless samples $|f(x_0+\frac{n}{2B})|$, $|f(x_0+\frac{n+1}{2B})-f(x_0+\frac{n}{2B})|$ and $|f(x_0+\frac{n+2}{2B})-f(x_0+\frac{n}{2B})|, n\in \Z$ for some $x_0\in \R$, up to a unimodular constant and conjugation. After we determine the samples $f(x_0+\frac{n}{2B})$, up to a unimodular and conjugation, we use the following formula to determine the signal, cf. WSK sampling theorem.

	\begin{lemma}\label{lem.WSK}
		Let $B>0$, the Paley-Wiener space $PW_B^2$ be defined as in (\ref{PW.def}). Assume $f\in PW_B^2$ and $x_0\in \R$, then
		\begin{align*}
			f(x) = \sum_{n\in \Z} f(x_0 + \frac{n}{2B}) \text{sinc} (2B(x-x_0)-n), x\in \R,
		\end{align*}
		where the series converges to $f$ in $L^2(\R)$.
	\end{lemma}
	\begin{proof}
		One may easily verify that $\mathcal{F}(sinc)= \textbf{1}|_{[-\frac{1}{2},\frac{1}{2}]}$ and $\mathcal{F}(sinc(2B\cdot))= \frac{1}{2B}\textbf{1}|_{[-B,B]}$.
		
		We can obtain that $\{\frac{1}{\sqrt{2B}} e^{2\pi iw (-x_0+ \frac{n}{2B})}\}_n$ is a orthonormal basis in $L^2([-B,B])$. In fact, for every $n,m\in \Z$
		\begin{align*}
			\int_{-B}^{B} e^{\frac{2\pi i(n-m)w}{2B}} dw {}& = \mathcal{F}^{-1} (\text{1}|_{[-B,B]}) (\frac{n-m}{2B}) = 2B \text{sinc} (2B \frac{n-m}{2B})\\
			{}& = 2B \text{sinc}(n-m) = 2B\delta_{m,n}.
		\end{align*}
		We have
		\begin{align*}
			\hat{f}(w) = \sum_{n\in \Z} a_n e^{2\pi iw (-x_0+\frac{n}{2B})},
		\end{align*}
		where
		\begin{align*}
			a_n = \frac{1}{2B} \int_{-B}^{B} \hat{f}(w) e^{- 2\pi iw (-x_0+\frac{n}{2B})} dw = \frac{1}{2B} \int_\R \hat{f}(w) e^{- 2\pi iw (-x_0+\frac{n}{2B})} dw = \frac{1}{2B} f(x_0 - \frac{n}{2B}),
		\end{align*}
		which implies
		\begin{align*}
			\hat{f}(w) = \frac{1}{2B} \sum_{n\in \Z} f(x_0 - \frac{n}{2B}) e^{2\pi iw (-x_0+\frac{n}{2B})} = \frac{1}{2B} \sum_{n\in \Z} f(x_0 + \frac{n}{2B}) e^{2\pi iw (-x_0-\frac{n}{2B})}.
		\end{align*}
		Then
		\begin{align*}
			f(x) {}& = \int_\R \hat{f} (w) e^{2\pi ixw} dw = \int_{-B}^B \hat{f} (w) e^{2\pi ixw} dw \\
			{}& = \frac{1}{2B} \int_{-B}^B \sum_{n\in \Z} f(x_0 + \frac{n}{2B}) e^{2\pi iw (-x_0-\frac{n}{2B})} e^{2\pi ixw} dw \\
			%{}& = \frac{1}{2B} \sum_{n\in \Z} f(x_0 + \frac{n}{2B}) \int_{-B}^B e^{2\pi iw (-x_0-\frac{n}{2B})} e^{2\pi ixw} dw \\
			{}& = \frac{1}{2B} \sum_{n\in \Z} f(x_0 + \frac{n}{2B}) \mathcal{F}^{-1} (\textbf{1}|_{[-B,B]}) (x-x_0- \frac{n}{2B}) \\
			%{}& = \sum_{n\in \Z} f(x_0 + \frac{n}{2B}) \text{sinc} (2B(x - x_0-\frac{n}{2B})) \\
			{}& = \sum_{n\in \Z} f(x_0 + \frac{n}{2B})  \text{sinc} (2B(x-x_0)-n).
		\end{align*}
	\end{proof}

\end{appendices}

\end{document}